\documentclass{amsart}
\usepackage{amssymb}
\usepackage{amsmath}
\makeindex

\newcommand{\ba}{\begin{array}}
\newcommand{\eea}{\end{eqnarray}}
\newcommand{\ea}{\end{array}}

\newcommand{\vare}{\varepsilon}

\newcommand{\I}{\mathbb I}
\newcommand{\D}{\mathbb D}

\newcommand{\R}{\mathbb R}%

\newtheorem{definition}{Definition}[section]
\newtheorem{theorem}[definition]{Theorem}
\newtheorem{lemma}[definition]{Lemma}
\newtheorem{proposition}[definition]{Proposition}
\newtheorem{corollary}[definition]{Corollary}
\newtheorem{example}[definition]{Example}
\newtheorem{remark}[definition]{Remark}

\newtheorem{observation}[definition]{Observation}

\usepackage[matrix,arrow,curve]{xy}
\usepackage[utf8x]{inputenc}
\usepackage{tikz}

\begin{document}
\title[Classification of contact foliations]{Homotopy classification of contact foliations\\ on open contact manifolds}
\author[M. Datta]{Mahuya Datta}
\address{Statistics and Mathematics Unit, Indian Statistical Institute\\ 203,
B.T. Road, Calcutta 700108, India.\\ e-mail:mahuya@isical.ac.in\\ }
\author[S. Mukherjee]{Sauvik Mukherjee}
\address{Statistics and Mathematics Unit, Indian Statistical Institute\\ 203,
B.T. Road, Calcutta 700108, India.\\
e-mail:mukherjeesauvik@yahoo.com}
\begin{abstract}We have given a homotopy classification of foliations on open contact manifolds whose leaves are contact submanifolds of the ambient space. The result is an extension of Haefliger's classification of foliations on open manifold. On the way to the main theorem we prove a result on equidimensional isocontact immersions on open contact manifolds.
\end{abstract}
\maketitle

\section{Introduction}

In \cite{haefliger}, Haefliger gave a homotopy classification of foliations on open manifolds.
A foliation on a manifold $M$ can be thought of as a partition of the manifold into injectively immersed submanifolds, called leaves. The simplest type of foliations on a manifold $M$ are obtained from submersions on it, in which case the level sets of the submersions are the leaves of foliations. More generally, if a smooth map $f:M\to N$ is transversal to a given foliation $\mathcal F_N$ on $N$ then the inverse image of $\mathcal F_N$ under $f$, denoted by $f^{-1}\mathcal F_N$, is a foliation on $M$; it is a standard fact that the codimension of $f^{-1}\mathcal F_N$ is the same as the codimension of $\mathcal F_N$. Using a result of Phillips on homotopy classification of transversal maps to foliations (\cite{phillips1}), Haefliger obtained a classification of foliations on open manifolds. Later on, this result was extended to all manifolds by Thurston.

In this article we shall study foliations on an open manifolds $M$ in the presence of contact form and extend the result of Haefliger.
Let $(M,\alpha)$ be a contact manifold with contact form $\alpha$. Then $\ker\alpha$ is a codimension 1 subbundle of the tangent bundle $TM$ and the restriction of $d\alpha$ to $\ker\alpha$ is a symplectic structure on the bundle. A foliation $\mathcal F$ on $M$ will be called a \emph{contact foliation on $M$ subordinate to} $\alpha$ (or simply a \emph{contact foliation} on $(M,\alpha)$) if the leaves of $\mathcal F$ are contact submanifolds of $M$. The tangent distribution $T\mathcal F$ of a contact foliation is transversal to the contact subbundle $\ker\alpha$; moreover, the intersection $T\mathcal F$ with $\ker\alpha$ is a symplectic subbundle of $\ker\alpha$ with respect to the symplectic structure $d'\alpha=d\alpha|_{\ker\alpha}$.

Suppose that $N$ is an arbitrary manifold with a foliation $\mathcal F_N$ of codimension $2q$ which is strictly less than $\dim M$. We shall denote the normal bundle $TM/T\mathcal F$ of $\mathcal F_n$ by $\nu\mathcal F_N$ and $\pi:TN\to\nu\mathcal F_N$ will denote the canonical projection map.
Let $Tr_\alpha(M,\mathcal F_N)$ be the space of all maps $f:M\rightarrow N$ which are transversal to $\mathcal F_N$ and such that $f^{-1}\mathcal F_N$, the inverse image of the foliation $\mathcal F_N$,  is a contact foliation on $M$. Let $\mathcal E_\alpha(TM,\nu\mathcal F_N)$ be the space of all vector bundle morphisms $F:TM\rightarrow TN$ such that
\begin{enumerate}
\item $\pi\circ F:TM\to \nu(\mathcal F_N)$ is an epimorphism,
\item $\ker(\pi\circ F)$ is transverse to the contact distribution $\ker\alpha$ and
\item $\ker(\pi\circ F)\cap\ker\alpha$ is a symplectic subbundle of $(\ker\alpha,d'\alpha)$.\end{enumerate}
With $C^{\infty}$-compact open topology on $Tr_\alpha(M,\mathcal F_N)$ and $C^{0}$-compact open topology on $\mathcal{E}_\alpha(TM,\nu\mathcal F_N)$ we obtain the following result.
\begin{theorem}\label{T:contact-transverse}
Let $(M,\alpha)$ be an open contact manifold and $(N,\mathcal F_N)$ be any foliated manifold. Suppose that the codimension of $\mathcal F_N$ is even and is strictly less than the  dimension of $M$. Then
\[\pi\circ d:Tr_\alpha(M,\mathcal F_N)\to\mathcal{E}_\alpha(TM,\nu\mathcal F_N)\]
is a weak homotopy equivalence.\end{theorem}
\noindent Theorem~\ref{T:contact-transverse} may be viewed as an extension of Phillips' Transversality Theorem (\cite{phillips1}) in the contact setting. Using this result we can obtain a homotopy classification of contact foliations on $(M,\alpha)$ following Haefilger (\cite{haefliger}). To state the result, let $\Gamma_q$ be the groupoid of germs of local diffeomorphisms of $\R^q$ and $B\Gamma_q$ be the classifying space of $\Gamma_q$ structures with the universal $\Gamma_q$-structure $\Omega_q$. The homotopy classes of $\Gamma_q$ structures on $M$ are in one-to-one correspondence with the the homotopy classes of continuous maps $M\to B\Gamma_q$ (see \cite{haefliger1}). Any $\Gamma_q$ structure on $M$ can be obtained as the pullback $f^*\Omega_q$ by a continuous map $f:M\to B\Gamma_q$. Theorem~\ref{T:contact-transverse} leads to the following classification of contact foliations on open contact manifolds.
\begin{theorem} Let $(M,\alpha)$ be an open contact manifold. The integral homotopy classes of codimension $2q$ contact foliations on $M$ subordinate to $\alpha$ are  in one-to-one correspondence with the `integrable homotopy' classes of bundle epimorphisms $(F,f):TM\to \nu\Omega_{2q}$ for which $\ker F\cap \ker\alpha$ is a symplectic subbundle of $\ker\alpha$.\label{MT1}\end{theorem}

\noindent Theorem~\ref{T:contact-transverse} will follow from a general $h$-principle type result (see Theorem~\ref{CT} stated below) by observing that $Tr_\alpha(M,\mathcal F_N)$ is the solution space of some open relation which is invariant under the action of local contactomorphisms.
\begin{theorem}
\label{CT}
Let $(M,\alpha)$ be an open contact manifold and $\mathcal{R}\subset J^{r}(M,N)$ be an open relation invariant under the action of the pseudogroup of local contactomorphisms of $(M,\alpha)$. Then the parametric $h$-principle holds for $\mathcal{R}$.
\end{theorem}
\noindent
A symplectic analogue of Theorem~\ref{CT} was proved in \cite{datta-rabiul} using a result of Ginzburg (\cite{ginzburg}). Ginzburg demonstrated some weaker form of stability for symplectic forms on open manifolds, though Moser's stability is known to be false on such manifolds. Here we prove a contact analogue of Ginzburg's result which can be stated as follows.
\begin{theorem}
Let $\xi_{t}$, $t\in[0,1]$ be a continuous family of contact structures defined by the contact forms $\alpha_t$ on a compact manifold $M$ with boundary. Let $(N,\tilde{\xi}=\ker\eta)$ be a contact manifold without boundary. Then every isocontact immersion $f_0:(M,\xi_0)\to (N,\tilde{\xi})$ admits a regular homotopy $\{f_t\}$ such that $f_t:(M,\xi_t)\to (N,\tilde{\xi})$ is an isocontact immersion for all $t\in[0,1]$.

In addition, if $M$ contains a compact submanifold $V_{0}$ in its interior and $\xi_{t}=\xi_{0}$ on $\it{Op}(V_{0})$ then $f_{t}$ can be chosen to be a constant homotopy on $Op\,(V_{0})$.\label{T:equidimensional_contact immersion}
\end{theorem}
As a corollary to Theorem~\ref{T:equidimensional_contact immersion} we show that every open contact manifold admits a regular homotopy of contact immersions $\varphi_t$, $t\in [0,1]$, such that $\varphi_0=id_M$ and $\varphi_1$ takes $M$ into an arbitrary neighbourhood of a core of $M$. This has a very important role to play in the proof of Theorem~\ref{CT}.

The paper is organised as follows. We recall preliminaries of contact manifolds in Section 2. Theorem~\ref{T:equidimensional_contact immersion} is proved in Section 3. In Section 4, we prove Theorem~\ref{CT} after briefly reviewing the language of $h$-principle and a few major results which are necessary for our purpose. We prove Theorem~\ref{T:contact-transverse}  and Theorem~\ref{MT1} in Sections 5 and 7 respectively. In the final section we give an example of contact foliation on some open subsets of odd-dimensional spheres. We include the relevant background of $\Gamma_q$-structures and its relations to foliations in Section 7.

\section{Preliminaries of contact manifolds}
In this section we review basic definitions and results related to contact manifolds.

\begin{definition}{\em
Let $M$ be a $2n+1$ dimensional manifold. A 1-form $\alpha$ on $M$ is said to be a \emph{contact form} if $\alpha \wedge (d\alpha)^n$ is nowhere vanishing. }\label{contact_form}
\end{definition}
If $\alpha$ is a contact form then
\[d'\alpha=d\alpha|_{\ker\alpha}\index{$d'\alpha$}\]
is a symplectic structure on the hyperplane distribution $\ker\alpha$. Also, there is a global vector field $R_\alpha$ on $M$ defined by the relations
\begin{equation}\alpha(R_\alpha)=1,\ \ \ i_{R_\alpha}.d\alpha=0,\label{reeb} \index{Reeb vector field}
\end{equation}
where $i_X$ denotes the interior multiplication by the vector field $X$. Thus, $TM$ has the following decomposition:
\begin{equation}TM=\ker\alpha \oplus \ker\,d\alpha,\label{decomposition}\end{equation}
where $\ker\alpha$ is a symplectic vector bundle and $\ker\,d\alpha$ is the 1-dimensional subbundle generated by $R_\alpha$. The vector field $R_\alpha$ is called the \emph{Reeb vector field} of the contact form $\alpha$.

A codimension 1 hyperplane distribution $\xi$ on $M$ is said to be a \emph{contact structure} on $M$ if $\xi$ is locally defined as the kernel of a (local) contact form $\alpha$. Observe that the local contact form in this case is defined uniquely up to multiplication by a nowhere vanishing function $f$. Moreover, $d(f\alpha)|_\xi=f. d\alpha|_\xi$ and hence every contact structure is associated with a conformal symplectic structure.

If $\alpha$ is a contact form then the distribution $\ker\alpha$ will be called the \emph{contact distribution of} $\alpha$.
\begin{example}\label{ex:contact}\end{example}
\begin{enumerate}\item
Every odd dimensional Euclidean space $\R^{2n+1}$ has a canonical contact form given by $\alpha=dz+\sum_{i=1}^nx_i\,dy_i$, where $(x_1,\dots,x_n,y_1,\dots,y_n,z)$ is the canonical coordinate system on $\R^{2n+1}$.
\item
Every even dimensional Euclidean space $\R^{2n}$ has a canonical 1-form $\lambda=\sum_{i=1}^n(x_idy_i-y_idx_i)$ which is called the Liouville form of $\R^{2n}$, where $(x_1,\dots,x_n$, $y_1,\dots,y_n)$ is the canonical coordinate system on $\R^{2n}$. The restriction of $\lambda$ on the unit sphere in $\R^{2n}$ defines a contact form.
\end{enumerate}

A contact form $\alpha$ also defines a canonical isomorphism $\phi:TM\to T^*M$ between the tangent and the cotangent bundles of $M$ given by
\begin{equation}\phi(X)=i_X d\alpha+\alpha(X)\alpha, \text{ for } X\in TM.\label{tgt_cotgt}\end{equation}
It is easy to see that the Reeb vector field $R_\alpha$ corresponds to the 1-form $\alpha$ under $\phi$.

\begin{definition} {\em Let $(N,\xi)$ be a contact manifold. A monomorphisn $F:TM\to (TN,\xi)$ is called \textit{contact} if $F$ is transversal to $\xi$ and $F^{-1}(\xi)$ is a contact structure on $M$. A smooth map $f:M\to (N,\xi)$ is called \textit{contact} if its differential $df$ is contact.

If $M$ is also a contact manifold with a contact structure $\xi_0$, then a monomorphism $F:TM\to TN$ is said to be \textit{isocontact} if $\xi_0=F^{-1}\xi$ and $F:\xi_0\to\xi$ is conformal symplectic with respect to the conformal symplectic structures on $\xi_0$ and $\xi$. A smooth map $f:M\to N$ is said to be \textit{isocontact} if $df$ is isocontact.

A diffeomorphism $f:(M,\xi)\to (N,\xi')$ is said to be a \emph{contactomorphism} \index{contactomorphism} if $f$ is isocontact.
\index{isocontact map}}
\end{definition}
If $\xi=\ker\alpha$ for a globally defined 1-form $\alpha$ on $N$, then $f$ is contact if $f^*\alpha$ is a contact form on $M$.
Furthermore, if $\xi_0=\ker\alpha_0$ then $f$ is isocontact if $f^*\alpha=\varphi \alpha_0$ for some nowhere vanishing function $\varphi:M\to\R$.

\begin{definition}{\em A vector field $X$ on a contact manifold $(M,\alpha)$ is called a \emph{contact vector field} if it satisfies the relaion $\mathcal L_X\alpha=f\alpha$ for some smooth function $f$ on $M$, where $\mathcal L_X$ denotes the Lie derivation operator with respect to $X$.}\label{D:contact_vector_field} \index{contact vector field}\end{definition}
Every smooth function $H$ on a contact manifold $(M,\alpha)$ gives a contact vector field $X_H=X_0+\bar{X}_H$ defined as follows:
\begin{equation}X_0=HR_\alpha \ \ \ \text{ and }\ \ \ \bar{X}_H\in \Gamma(\xi) \text{ such that  }i_{\bar{X}_H}d\alpha|_\xi = -dH|_\xi,\label{contact_hamiltonian} \index{contact hamiltonian}\end{equation}
where $\xi=\ker\alpha$; equivalently,
\begin{equation}\alpha(X_H)=H\ \ \text{ and }\ \ i_{X_H}d\alpha=-dH+dH(R_{\alpha})\alpha.\label{contact_hamiltonian1} \end{equation}
The vector field $X_H$ is called the \emph{contact Hamiltonian vector field} of $H$.

If $\phi_t$ is a local flow of a contact vector field $X$, then
\[\frac{d}{dt}\phi_t^*\alpha = \phi_t^*(i_X.d\alpha+d(\alpha(X)))=\phi_t^*(f\alpha)=(f\circ\phi_t)\phi_t^*\alpha.\]
Therefore, $\phi_t^*\alpha=\lambda_t\alpha$, where $\lambda_t=e^{\int f\circ\phi_t \,dt}$. Thus the flow of a contact vector field preserves the contact structure.

\begin{theorem}(Gray's Stability Theorem (\cite{gray}))
 \label{T:gray stability}
If $\xi_t,\ t\in \I$ is a smooth family of contact structures on a closed manifold $M$, then there exists an isotopy $\psi_t,\ t\in \I$, of $M$ such that \[\psi_t:(M,\xi_0)\to (M,\xi_t) \]
is isocontact for all $t\in \I$
\end{theorem}

\begin{remark}{\em
Gray's stability theorem is not valid on non-closed manifolds. We shall see an extension of Theorem~\ref{T:gray stability} for such manifolds in Theorem~\ref{T:equidimensional_contact immersion}) which is one of the main results of this article.}
\end{remark}
We end this section with the definition of a contact submanifold.
\begin{definition} {\em A submanifold $N$ of a contact manifold $(M,\xi)$ is said to be a \emph{contact submanifold} if the inclusion map $i:N\to M$ is a contact map.}\label{contact_submanifold} \index{contact submanifold}\end{definition}

\begin{lemma} A submanifold $N$ of a contact manifold $(M,\xi=\ker\alpha)$ is a contact submanifold if and only if $TN$ is transversal to $\xi|_N$ and $TN\cap\xi|_N$ is a symplectic subbundle of $(\xi,d'\alpha)$.\label{L:contact_submanifold}\end{lemma}

\section{Equidimensional contact immersions}

We begin with a simple observation.
\begin{observation}{\em
Let $(M,\alpha)$ be a contact manifold. The product manifold $M\times\R^2$ has a canonical contact form given by $\tilde{\alpha}=\alpha- y\,dx$, where $(x,y)$ are the coordinate functions on $\R^2$. We shall denote the contact structure associated with $\tilde{\alpha}$ by $\tilde{\xi}$.
Now suppose that $H:M\times\R\to \R$ is a smooth function which vanishes on some open set $U$. Define $\bar{H}:M\times\R\to M\times\R^2$ by $\bar{H}(u,t)=(u,t,H(u,t))$ for all $(u,t)\in M\times\R$. Since $\bar{H}(u,t)=(u,t,0)$ for all $(u,t)\in U$, the image of $d\bar{H}_{(u,t)}$ is $T_uM\times\R\times\{0\}$. On the other hand,  $\tilde{\xi}_{(u,t,0)}=\xi_u\times\R^2$. Therefore, $\bar{H}$ is transversal to $\tilde{\xi}$ on $U$.
}\end{observation}

\begin{proposition} Let $M$ be a contact manifold with contact form $\alpha$. Suppose that $H$ is a smooth real-valued function on $M\times(-\vare,\vare)$ with compact support
such that its graph $\Gamma$ in $M\times\R^2$ is transversal to the kernel of $\tilde{\alpha}=\alpha-y\,dx$. Then there is a diffeomorphism $\Psi:M\times(-\vare,\vare)\to \Gamma$ which pulls back $\tilde{\alpha}|_{\Gamma}$ onto $h(\alpha\oplus 0)$, where $h$ is a nowhere-vanishing smooth real-valued function on $M\times\R$.\label{characteristic}
\end{proposition}
\begin{proof} Since the graph $\Gamma$ of $H$ is transversal to $\tilde{\xi}$, the restriction of $\tilde{\alpha}$ to $\Gamma$ is a nowhere vanishing 1-form on it.
Define a function $\bar{H}:M\times(-\vare,\vare)\to M\times\R^2$ by $\bar{H}(u,t)=(u,t,H(u,t))$. The map $\bar{H}$ defines a diffeomorphism of $M\times(-\vare,\vare)$ onto $\Gamma$, which pulls back the form $\tilde{\alpha}|_{\Gamma}$ onto $\alpha-H\,dt$. It is therefore enough to obtain a diffeomorphism $F:M\times(-\vare,\vare)\to M\times(-\vare,\vare)$ which pulls back the 1-form $\alpha-H\,dt$ onto a multiple of $\alpha\oplus 0$.
For each $t$, define a smooth function $H^t$ on $M$ by $H^t(u)=H(u,t)$ for all $u\in M$. Let
$X_{H^t}$ denote the contact Hamiltonian vector field on $M$ associated with $H^t$. Consider the vector field
$\bar{X}$ on $M\times\R$ as follows: \[\bar{X}(u,t)=(X_{H^t}(u),1),\ \ (u,t)\in M\times(-\vare,\vare).\]
Let $\{\bar{\phi}_s\}$ denote a local flow of $\bar{X}$ on $M\times\R$. Then writing $\bar{\phi}_s(u,t)$ as
\begin{center}$\bar{\phi}_s(u,t)=(\phi_s(u,t),s+t)$ for all $u\in M$ and $s,t\in \R$,\end{center}
we get the following relation:
\[\frac{d\phi_s}{ds}(u,t)=X_{t+s}(\phi_s(u,t)),  \]
where $X_t$ stands for the vector field $X_{H^t}$ for all $t$. In particular, we have
\begin{equation}\frac{d\phi_t}{dt}(u,0)=X_t(\phi_t(u,0)),\label{flow_eqn}\end{equation}
Define a level preserving map $F:M\times(-\vare,\vare)\to M\times(-\vare,\vare)$ by
\[F(u,t)=\bar{\phi}_t(u,0)=(\phi_t(u,0),t).\]
Since the support of $H$ is contained in $K\times (-\vare,\vare)$ for some compact set $K$, the flow $\bar{\phi}_s$ starting at $(u,0)$ remains within $M\times (-\vare,\vare)$ for $s\in (-\vare,\vare)$.
Note that
\begin{center}$dF(\frac{\partial}{\partial t})=\frac{\partial}{\partial t}\bar{\phi}_t(u,0)=\bar{X}(\bar{\phi}_t(u,0))=\bar{X}(\phi_t(u,0),t)=(X_{H^t}(\phi_t(u,0)),1).$\end{center}
This implies that
\begin{center}$\begin{array}{rcl}F^*(\alpha\oplus 0) (\frac{\partial}{\partial t}|_{(u,t)}) &  = & (\alpha\oplus 0)(dF(\frac{\partial}{\partial t}|_{(u,t)}))\\
& = & \alpha(X_{H^t}(\phi_t(u,0)))\\
& = &H^t(\phi_t(u,0))\ \ \ \ \ \text{by equation }(\ref{contact_hamiltonian1}) \\
& = & H(\bar{\phi}_t(u,0))=H(F(u,t)) \end{array}$\end{center}
Also,
\begin{center}$F^*(H\,dt)(\frac{\partial}{\partial t})=(H\circ F)\,dt(dF(\frac{\partial}{\partial t}))=H\circ F$\end{center}
Hence,
\begin{equation}F^*(\alpha-Hdt)(\frac{\partial}{\partial t})=0.\label{eq:F1}\end{equation}
On the other hand,
\begin{equation}F^*(\alpha-H\,dt)|_{M\times\{t\}}=F^*\alpha|_{M\times\{t\}}=\psi_t^*\alpha,\label{eq:F2}\end{equation}
where $\psi_t(u)=\phi_t(u,0)$, $\psi_0(u)=u$. Thus, $\{\psi_t\}$ are the integral curves of the time dependent vector field $\{X_t\}$ on $M$ (see (\ref{flow_eqn})), and we get
\[\begin{array}{rcl}\frac{d}{dt}\psi_t^*\alpha & = & \psi^*_t(i_{X_t}d\alpha+d(i_{X_t}\alpha))\\
& = & \psi^*_t(dH^t(R_\alpha)\alpha-dH^t+dH^t) \ \ \text{by equation }(\ref{contact_hamiltonian1})\\
& = & \psi^*_t(dH^t(R_\alpha)\alpha)\\
& = & \theta(t)\psi_t^*\alpha,\end{array}\]
where $\theta(t)=\psi_t^*(dH^t(R_\alpha))$. Hence $\psi_t^*\alpha=e^{\int_0^t\theta(s)ds}\psi_0^*\alpha=e^{\int_0^t\theta(s)ds}\alpha$. We conclude from equation (\ref{eq:F1}) and (\ref{eq:F2}) that $F^*(\alpha-H\,dt)=e^{\int_0^t\theta(s)ds}\alpha$. Finally, take $\Psi=\bar{H}\circ F$ which has the desired properties.
\end{proof}
\begin{remark}{\em
If there exists an open subset $\tilde{U}$ of $M$ such that $H$ vanishes on $\tilde{U}\times (-\vare,\vare)$ then the contact Hamiltonian vector fields $X_t$ defined above are identically zero on $\tilde{U}$ for all $t\in(-\vare,\vare)$. Since $\psi_t=\phi_t(\ \ ,0)$ are the integral curves of the time dependent vector fields $X_t=X_{H^t}$, $0\leq t\leq 1$, we must have $\psi_t(u)=u$ for all $u\in \tilde{U}$. Therefore, $F(u,t)=(u,t)$ and hence $\Psi(u,t)=(u,t,0)$ for all $u\in\tilde{U}$ and all $t\in(-\vare,\vare)$.}\label{R:characteristic}\end{remark}

\begin{remark}{\em
If $\Gamma$ is a codimension 1 submanifold of a contact manifold $(N,\tilde{\alpha})$ such that the tangent planes of $\Gamma$ are transversal to $\tilde{\xi}=\ker\tilde{\alpha}$ then there is a codimension 1 distribution $D$ on $\Gamma$ given by the intersection of $\ker \tilde{\alpha}|_\Gamma$ and $T\Gamma$. Since $D=\ker\tilde{\alpha}|_\Gamma\cap T\Gamma$ is an odd dimensional distribution, $d\tilde{\alpha}|_D$ has a 1-dimensional kernel. If $\Gamma$ is locally defined by a function $\Phi$ then $d\Phi_x$ does not vanish identically on $\ker\tilde{\alpha}_x$, for $\ker d\Phi_x$ is transversal to $\ker\tilde{\alpha}_x$. Thus there is a unique non-zero vector $Y_x$ in $\ker\tilde{\alpha}_x$ satisfying the relation $i_{Y_x}d\alpha_x=d\Phi_x$. Clearly, $Y_x$ is tangent to $\Gamma$ at $x$ and it is defined uniquely only up to multiplication by a non-zero real number (as $\Phi$ is not unique). However, the 1-dimensional distribution on $\Gamma$ defined by $Y$ is uniquely defined by the contact form $\tilde{\alpha}$. The integral curves of $Y$ are called \emph{characteristics} of $\Gamma$ (\cite{arnold}).

It can be verified that the diffeomorphism $\Psi$ in the proof of the above proposition maps the lines in $M\times\R$ onto the characteristics on $\Gamma$.}
\end{remark}
The following lemma will reduce Theorem~\ref{T:equidimensional_contact immersion} to the special case in which the contact forms $\alpha_t$ are piecewise primitive. A non-parametric form of this lemma was proved in \cite{eliashberg}.
\begin{lemma}
\label{EM2} Let $\alpha_{t}, t\in[0,1]$, be a continuous family of contact forms on a compact manifold $M$, possibly with non-empty boundary. Then for each $t\in [0,1]$, there exists a sequence of primitive 1-forms $\beta_t^l=r_t^l\,ds_t^l, l=1,..,N$ such that
\begin{enumerate}
\item $\alpha_t=\alpha_0+\sum_1^N \beta_t^l$ for all $t\in [0,1]$,
\item for each $j=0,..,N$ the form $\alpha^{(j)}_t=\alpha_{0}+\sum_{1}^{j}\beta_t^{l}$ is contact,
\item for each $j=1,..,N$ the functions $r_t^j$ and $s_t^j$ are compactly supported within a coordinate neighbourhood.
\end{enumerate}
Furthermore, the forms $\beta_t^l$ depends continuously on $t$.

If $\alpha_t=\alpha_0$ on $Op\,V_0$, where $V_0$ is a compact subset contained in the interior of $M$, then the functions $r^l_t$ and $s^l_t$ can be chosen to be equal to zero on an open neighbourhood of $V_0$.
\end{lemma}
\begin{proof} If $M$ is compact and with boundary, then we embed it in a bigger manifold $\tilde{M}$ of the same dimension; in fact, we may assume that $\tilde{M}$ is obtained from  $M$ by attaching a collar along the boundary of $M$.
Using the compactness property of $M$, one can cover $M$ by finitely many coordinate neighbourhoods $U^i, i=1,2,\dots,L$.
Choose a partition of unity $\{\rho^i\}$ subordinate to $\{U^i\}$.
\begin{enumerate}\item Since $M$ is compact, the set of all contact forms on $M$ is an open subspace of $\Omega^1(M)$ in the weak topology. Hence, there exists a $\delta>0$ such that $\alpha_t+s(\alpha_{t'}-\alpha_t)$ is contact for all $s\in[0,1]$, whenever $|t-t'|<\delta$.
\item Get an integer $n$ such that $1/n<\delta$.
Define for each $t$ a finite sequence of contact forms, namely $\alpha^j_t$, interpolating between $\alpha_0$ and $\alpha_t$ as follows:
\[\alpha^j_t=\alpha_{[nt]/n}+\sum_{i=1}^j\rho^i(\alpha_t-\alpha_{[nt]/n}),\]
where $[x]$ denotes the largest integer which is less than or equal to $x$ and $j$ takes values $1,2,\dots,L$. In particular, for $k/n\leq t\leq (k+1)/n$, we have
\[\alpha^j_t=\alpha_{k/n}+\sum_{i=1}^j\rho^i(\alpha_t-\alpha_{k/n}),\]
and $\alpha^L_t=\alpha_t$ for all $t$.
\item Let $\{x_j^i:j=1,\dots,m\}$ denote the coordinate functions on $U^i$, where $m$ is the dimension of $M$. There exists unique set of smooth functions $y_{t,k}^{ij}$ defined on $U^i$ satisfying the following relation:
\[\alpha_t-\alpha_{k/n}=\sum_{j=1}^m y_{t,k}^{ij} dx^i_j\ \ \text{on } U^i \text{ for }k/n\leq t\leq (k+1)/n\]
Further, note that $y_{t,k}^{ij}$ depends continuously on the parameter $t$ and $y_{t,k}^{ij}=0$ when $t=k/n$, $k=0,1,\dots,n$.
\item Let $\sigma^i$ be a smooth function such that $\sigma^i\equiv 1$ on a neighbourhood of $\text{supp\,}\rho^i$ and $\text{supp}\,\sigma^i\subset U^i$. Define functions $r^{ij}_{t,k}$ and $s^{ij}$, $j=1,\dots,m$, as follows:
        \[ r_{t,k}^{ij}=\rho^i y_t^{ij} \ \ \ s^{ij}=\sigma^i x^i_j.\]
These functions are compactly supported and supports are contained in $U^i$. It is easy to see that $r^{ij}_{t,k}=0$ when $t=k/n$ and
\[\rho^i(\alpha_t-\alpha_{k/n})=\sum_{j=1}^m r_{t,k}^{ij}\,ds^{ij} \ \text{ for }\ t\in [k/n,(k+1)/n].\]
\end{enumerate}
It follows from the above discussion that $\alpha_t-\alpha_{k/n}$ can be expressed as a sum of primitive forms which depends continuously on $t$ in the interval $[k/n,(k+1)/n]$.
We can now complete the proof by finite induction argument. Suppose that
$(\alpha_{t}-\alpha_0)=\sum_l\alpha_{t,k}^l$ for $t\in [0,k/n]$, where each $\alpha_{t,k}^l$ is a primitive 1-form. Define
\[\begin{array}{rcl}\tilde{\alpha}_{t,k}^l& = & \left\{
\begin{array}{cl} \alpha_{t,k}^l & \text{if } t\in [0,k/n]\\
\alpha_{k/n,k}^l & \text{if } t\in [k/n,(k+1)/n]\end{array}\right.\end{array}\]
Further define for $j=1,\dots,N$, $i=1,\dots,L$,
\[\begin{array}{rcl}\beta_{t,k}^{ij}& = & \left\{
\begin{array}{cl} 0 & \text{if } t\in [0,k/n]\\
r^{ij}_{t,k}\,ds^{ij} & \text{if } t\in [k/n,(k+1)/n]\end{array}\right.\end{array}\]
Finally note that for $t\in [0,(k+1)/n]$, we can write $\alpha_t-\alpha_0$ as the sum of all the above primitive forms. Indeed, if $k/n\leq t<(k+1)/n$, then
\begin{eqnarray*}\alpha_t-\alpha_0 & = & (\alpha_t-\alpha_{k/n})+(\alpha_{k/n}-\alpha_0)\\
& = & \sum_{i=1}^L\sum_{j=1}^m r^{ij}_{t,k}\,ds^{ij}+\sum_l \alpha^l_{k/n,k}\\
& = & \sum_{i,j}\beta^{ij}_{t,k}+\sum_l \tilde{\alpha}^l_{t,k}.\end{eqnarray*}
The same relation holds for $0\leq t\leq k/n$, since $\beta^{ij}_{t,k}$ vanish for all such $t$. This proves the first part of the lemma.

Now suppose that $\alpha_t=\alpha_0$ on an open neighbourhood $U$ of $V_0$. Choose two compact neighbourhoods of $V_0$, namely $K_0$ and $K_1$ such that $K_0\subset \text{Int\,}K_1$ and $K_1\subset U$. Since $M\setminus\text{Int\,}K_1$ is compact we can cover it by finitely many coordinate neighbourhoods $U^i$, $i=1,2,\dots,L$, such that $(\bigcup_{i=1}^L U^i)\cap K_0=\emptyset$. Proceeding as above we get a decomposition of $\alpha_t$ on $\bigcup_{i=1}^L U^i$ into primitive 1-forms $r^l_t\,ds^l_t$. Observe that  $\{U^i:i=1,\dots,L\}\cup\{U\}$ is an open covering of $M$ in this case. The functions $r^l_t$ and $s^l_t$ can be extended to all of $M$ without disturbing their supports. Hence, the functions $r^l_t$ and $s^l_t$ vanish on $K_0$. This completes the proof of the lemma.

\end{proof}



\emph{Proof of Theorem~\ref{T:equidimensional_contact immersion}}. In view of Lemma~\ref{EM2}, it is enough to prove the theorem for a family of contact forms $\alpha_t$, $t\in [0,1]$, satisfying
\[\alpha_{t}=\alpha_{0}+r_tds_t \label{primitive_contact_forms}\]
for some smooth real valued functions $r_t,s_t$ which are (compactly) supported in an open set $U$ of $M$. We shall first show that $f_{0}:(M,\xi_{0})\rightarrow (N,\tilde{\xi})$ can be homotoped to an immersion $f_{1}:M\to N$ such that $f_{1}^{*}\tilde{\xi}=\xi_{1}$ which is a non-parametric version of the stated result.

For simplicity of notation we write $(r,s)$ for $(r_1,s_1)$ and define a smooth embedding $\varphi:U\to U\times\R^2$ by
\[\varphi(u)=(u,s(u),-r(u)) \text{ \ for \ }u\in U.\]
Since $r,s$ are compactly supported $\varphi(u)=(u,0,0)$ for all $u\in Op\,(\partial U)$ and there exist positive constants $\vare_1$ and $\vare_2$ such that $Im\,f$ is contained in $U\times I_{\vare_1}\times I_{\vare_2}$, where $I_\vare$ denotes the open interval $(-\vare,\vare)$ for $\vare>0$. Clearly, $\varphi^*(\alpha_0-y\,dx)=\alpha_0+r\,ds$ and so
\begin{equation}\varphi:(U,\xi_{1})\rightarrow (U\times \R^2,\ker(\alpha_{0}- y\,dx)) \label{eq:equidimensional_1}\end{equation}
is an isocontact embedding.
The image of $\varphi$ is the graph of a smooth function $k=(s,-r):U\rightarrow I_{\varepsilon_{1}}\times I_{\varepsilon_{2}}$ which is compactly supported with support contained in the interior of $U$. Further note that $\pi(\varphi(U))$ is the graph of $s$ and hence a submanifold of $U\times I_{\varepsilon_1}$. Now let $\pi:U\times I_{\varepsilon_{1}}\times I_{\varepsilon_{2}}\rightarrow U\times I_{\varepsilon_{1}}$ be the projection onto the first two coordinates. Since Im\,$\varphi$ is the graph of $k$,  $\pi|_{\text{Im\,}\varphi}$ is an embedding onto the set $\pi(\varphi(U))$ which is the graph of $s$. Now observe that Im\,$\varphi$ can also be viewed as the graph of a smooth function, namely $h:\pi(\varphi(U))\rightarrow I_{\varepsilon_{2}}$ defined by $h(u,s(u))=-r(u)$. It is easy to see that $h$ is compactly supported.

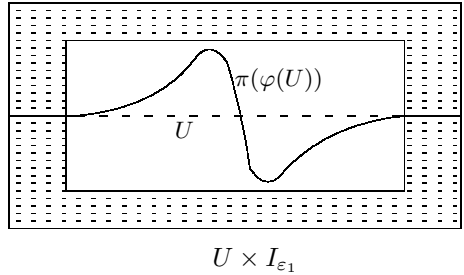
\begin{center}
\begin{picture}(300,140)(-100,5)\setlength{\unitlength}{1cm}
\linethickness{.075mm}

\multiput(-1,1.5)(6,0){2}
{\line(0,1){3}}
\multiput(-.25,2)(4.5,0){2}
{\line(0,1){2}}

\multiput(-1,1.5)(0,3){2}
{\line(1,0){6}}
\multiput(-.25,2)(0,2){2}
{\line(1,0){4.5}}
\put(1.7,1){$U\times I_{\varepsilon_{1}}$}
\put(1.2,2.7){\small{$U$}}
\put(2,3.4){\small{$\pi(\varphi(U))$}}

\multiput(-.9,1.6)(.2,0){30}{\line(1,0){.05}}
\multiput(-.9,1.7)(.2,0){30}{\line(1,0){.05}}
\multiput(-.9,1.8)(.2,0){30}{\line(1,0){.05}}
\multiput(-.9,1.9)(.2,0){30}{\line(1,0){.05}}

\multiput(-.9,4.0)(0,-.1){21}{\line(1,0){.05}}
\multiput(-.7,4.0)(0,-.1){21}{\line(1,0){.05}}
\multiput(-.5,4.0)(0,-.1){21}{\line(1,0){.05}}
\multiput(-.3,4.0)(0,-.1){21}{\line(1,0){.05}}

\multiput(4.3,4.0)(0,-.1){21}{\line(1,0){.05}}
\multiput(4.5,4.0)(0,-.1){21}{\line(1,0){.05}}
\multiput(4.7,4.0)(0,-.1){21}{\line(1,0){.05}}
\multiput(4.9,4.0)(0,-.1){21}{\line(1,0){.05}}

\multiput(-.9,4.1)(.2,0){30}{\line(1,0){.05}}
\multiput(-.9,4.2)(.2,0){30}{\line(1,0){.05}}
\multiput(-.9,4.3)(.2,0){30}{\line(1,0){.05}}
\multiput(-.9,4.4)(.2,0){30}{\line(1,0){.05}}

\multiput(-1,3)(.3,0){20}{\line(1,0){.1}}

\multiput(-1,3)(5.1,0){2}{\line(1,0){.9}}
\qbezier(-.1,3)(1,3.1)(1.5,3.8)
\qbezier(1.5,3.8)(1.7,4)(1.9,3.7)
\qbezier(1.9,3.7)(2.1,3)(2.2,2.3)
\qbezier(2.2,2.3)(2.4,2)(2.6,2.2)
\qbezier(2.6,2.2)(3.2,2.9)(4.3,3)

\end{picture}\end{center}
In the above figure, the bigger rectangle represents the set $U\times I_{\varepsilon_{1}}$ and the central dotted line represents $U\times 0$. The curve within the rectangle stands for the domain of $h$, which is also the graph of $s$. We can now extend $h$ to a compactly supported function $H:U\times I_{\varepsilon_{1}}\rightarrow I_{\varepsilon_{2}}$ (see \cite{whitney}) which vanishes on the shaded region and is such that its graph is transversal to $\ker(\alpha_0- y\,dx)$. Indeed, since $\varphi$ is an isocontact embedding it is transversal to $\ker(\alpha_0-y\,dx)$ and hence graph $H$ is transversal to $\ker(\alpha_0-y\,dx)$ on an open neighbourhood of $\pi(\varphi(U))$ for any extension $H$ of $h$. Since transversality is a generic property, we can assume (possibly after a small perturbation) that graph of $H$ is transversal to $\ker(\alpha_0- y\,dx)$.

Let $ \Gamma $ be the graph of $H$; then the image of $\varphi$ is contained in $\Gamma$. By Lemma~\ref{characteristic} there exists a diffeomorphism $\Phi:\Gamma\to U\times I_{\vare_1}$ with the property that
\begin{equation}\Phi^*(\ker(\alpha_{0}\oplus 0))=\ker((\alpha_{0}- y\,dx)|_\Gamma).\label{eq:equidimensional_2}\end{equation}
Next we use $f_0$ to define an immersion $F_{0}:U\times \mathbb R\rightarrow N\times \mathbb{R}$
as follows:
\begin{center} $F_0(u,x)=(f_0(u),x)$ for all $u\in U$ and $x\in \R$.\end{center}
It is straightforward to see that
\begin{itemize}
\item $F_{0}(u,0)\in N \times 0$ for all $u\in U$ and
\item $F_0^*(\eta \oplus 0)$ is a multiple of $\alpha_{0}\oplus 0$ by a nowhere vanishing function on $M\times\R$.
\end{itemize}
Therefore, the following composition is defined:
\[U\stackrel{\varphi}{\longrightarrow}  \Gamma\stackrel{\Phi}{\longrightarrow} U\times I_{\vare_1} \stackrel{F_0}{\longrightarrow} N\times\mathbb{R} \stackrel{\pi_{N}}{\longrightarrow} N, \]
where $\pi_{N}:N\times \mathbb{R}\rightarrow N$ is the projection onto $N$. Observe that $\pi_{N}^*\eta=\eta \oplus 0$ and therefore, it follows from equations (\ref{eq:equidimensional_1}) and (\ref{eq:equidimensional_2}) that the composition map $f_1=\pi_{N} F_{0}\Phi \varphi:(U,\xi_1)\rightarrow (N,\tilde{\xi})$ is isocontact. Such a map is necessarily an immersion.

Let $K=(\text{supp\,}r\cup\text{supp\,}s)$. Take a compact set $K_1$ in $U$ such that $K\subset \text{Int\,}K_1$, and let $\tilde{U}=U\setminus K_1$. If $u\in \tilde{U}$ then $\varphi(u)=(u,0,0)$. This gives $h(u,0)=0$ for all $u\in\tilde{U}$. We can choose $H$ such that $H(u,t)=0$ for all $(u,t)\in\tilde{U}\times I_{\vare_1}$. Then, by Remark~\ref{R:characteristic}, $\Phi(u,0,0)=(u,0)$ for all $u\in\tilde{U}$. Consequently,
\[f_1(u)=\pi_{N} F_{0}\Phi \varphi(u)=\pi_{N} F_{0}(u,0)=\pi_N(f_0(u),0)=f_0(u) \ \text{for all } u\in\tilde{U}.\]
In other words, $f_1$ coincides with $f_0$ outside an open neighbourhood of $K$.

Now, if we have a continuous family of contact forms $\alpha_t$ as in equation (\ref{primitive_contact_forms}) then define
\[\varphi_t(u)=(u,s_t(u),-r_t(u)) \text{ \ for \ }u\in U.\] Since each $\varphi_t$ has compact support, it follows that $\cup_{t\in[0,1]}\varphi_t(U)$ is a compact subset of $U\times\R^2$ and there exist positive constants $\vare_1$ and $\vare_2$ such that $\varphi_t(U)\subset U\times I_{\varepsilon_{1}}\times I_{\varepsilon_{2}}$ for all $t\in [0,1]$.  Proceeding exactly as before we get a continuous family of smooth functions $H_t$ such that their graphs $\Gamma_t$ are transversal to $\ker(\alpha_0-y\,dx)$. By applying Proposition~\ref{characteristic} we then get a continuous family of homeomorphisms $\Phi_t:\Gamma_t\to U\times I_{\vare_1}$ which pull back the $\ker(\alpha_{0}\oplus 0)$ onto $\ker(\alpha_{0}- y\,dx)|_{\Gamma_t}$. The desired homotopy $f_t$ is then defined by $f_t=\pi_{N} F_{0}\Phi_t \varphi_t$.
This completes the proof of the theorem.\qed

\begin{remark}{\em A symplectic version of the above result was proved by Ginzburg in \cite{ginzburg}.}\end{remark}

Every open manifold admits a Morse function $f$ without a local maxima. The codimension of the Morse complex of such a function is, therefore, strictly positive (\cite{milnor_morse},\cite{milnor}). The gradient flow of $f$ brings the manifold into an arbitrary small neighbourhood of the Morse complex. In fact, one can get a polyhedron $K\subset M$ such that codim\,$K>0$, and an isotopy $\phi_t:M\to M$, $t\in[0,1]$, such that $K$ remains pointwise fixed and $\phi_1$ takes $M$ into an
arbitrarily small neighborhood $U$ of $K$. The polyhedron $K$ is called a \emph{core} of $M$.
We shall now deduce from the above theorem, the existence of isocontact immersions of an open manifold $M$ into itself which compress the manifold $M$ into an arbitrary small neighbourhoods of its core.

\begin{corollary}
\label{CO}
Let $(M,\xi=\ker\alpha)$ be an open contact manifold and let $K$ be a core of it. Then for a given neighbourhood $U$ of $K$ in $M$ there exists a homotopy of isocontact immersions $f_{t}:(M,\xi)\rightarrow (M,\xi), t\in[0,1]$, such that $f_{0}=id_{M}$ and $f_{1}(M)\subset U$.
\end{corollary}
\begin{proof}Since $K$ is a core of $M$ there is an isotopy $g_t$ such that $g_0=id_M$ and $g_1(M)\subset U$.
Using $g_t$, we can express $M$ as $M=\bigcup_{0}^{\infty}V_{i}$, where $V_{0}$ is a compact neighbourhood of $K$ in $U$ and $V_{i+1}$ is diffeomorphic to $V_i\bigcup (\partial V_{i}\times [0,1])$ so that $\bar{V_i}\subset \text{Int}\,(V_{i+1})$ and $V_{i+1}$ deformation retracts onto $V_{i}$. If $M$ is a manifold with boundary then this sequence is finite. We shall inductively construct a homotopy of immersions $f^{i}_{t}:M\rightarrow M$ with the following properties:
\begin{enumerate}
\item $f^i_0=id_M$
\item $f^i_1(M)\subset U$
\item $f^i_t=f^{i-1}_t$ on $V_{i-1}$
\item $(f^i_t)^*\xi=\xi$ on $V_{i}$.
\end{enumerate}
Assuming the existence of $f^{i}_{t}$, let $\xi_{t}=(f^{i}_{t})^{*}(\xi)$ (so that $\xi_0=\xi$, and consider a 2-parameter family of contact structures defined by $\eta_{t,s}=\xi_{t(1-s)}$. Then for all $t,s\in\I$, we have:
\[\eta_{t,0}=\xi_t,\ \ \eta_{t,1}=\xi_0=\xi\ \text{ and }\ \eta_{0,s}=\xi.\]
The parametric version of Theorem~\ref{T:equidimensional_contact immersion} gives a homotopy of immersions $\tilde{f}_{t,s}:V_{i+2}\rightarrow M$, $(t,s)\in \I\times\I$, satisfying the following conditions:
\begin{enumerate}
 \item $\tilde{f}_{t,0},\tilde{f}_{0,s}:V_{i+2}\hookrightarrow M$ are the inclusion maps
 \item $(\tilde{f}_{t,s})^*\xi_t=\eta_{t,s}$; in particular, $(\tilde{f}_{t,1})^*\xi_t=\xi$
 \item $\tilde{f}_{t,s}=id$ on $V_i$ since $\eta_{t,s}=\xi_0$ on $V_i$.
\end{enumerate}
We now extend the homotopy $\{\tilde{f}_{t,s}|_{V_{i+1}}\}$ to all of $M$ as immersions such that $\tilde{f}_{0,s}=id_M$ for all $s$. By an abuse of notation, we denote the extended homotopy by the same symbol. Define the next level homotopy as follows:
\[f^{i+1}_{t}=f^{i}_{t}\circ \tilde{f}_{t,1} \ \text{ for }\ t\in [0,1].\]
This completes the induction step since $(f^{i+1}_t)^*(\xi)=(\tilde{f}_{t,1})^*\xi_t=\xi$ on $V_{i+2}$ for all $t$, and $f^{i+1}_t|_{V_{i}}=f^i_t|_{V_{i}}$.
To start the induction we use the isotopy $g_t$ and let $\xi_t=g_t^*\xi$. Note that $\xi_t$ is a family of contact structures on $M$ defined by contact forms $g_t^*\alpha$.
For starting the induction we construct $f^{0}_{t}$ as above by setting $V_{-1}=\emptyset$.

Having constructed the family of homotopies $\{f^i_t\}$ as above we set $f_t=\lim_{i\to\infty}f^i_t$ which is the desired homotopy of isocontact immersions.

\end{proof}

\section{An $h$-principle for open relations on open contact manifolds}
We begin with a brief exposition to the theory of $h$-principle. For further details we refer to \cite{gromov_pdr}.
Suppose that $M$ and $N$ are smooth manifolds. Let $J^{r}(M,N)$ be the space of $r$-jets of germs of local maps from $M$ to $N$ (\cite{golubitsky}). The canonical map $p^{(r)}:J^r(M,N)\to M$ which takes a jet $j^r_f(x)$ onto the base point $x$ is a fibration. We shall refer to $J^r(M,N)$ as the $r$-jet bundle over $M$. A continuous map $\sigma:M\to J^r(M,N)$ is said to be a \emph{section} of the jet bundle $p^{(r)}:J^r(M,N)\to M$ if $p^{(r)}\circ \sigma=id_M$. A section of $p^{(r)}$ which is the $r$-jet of some map $f:M\to N$ is called a \emph{holonomic} section of the jet bundle.

A subset $\mathcal{R} \subset J^{r}(M,N)$ of the $r$-jet space is called a \emph{partial differential relation of order} $r$ (or simply a \emph{relation}). If $\mathcal{R}$ is an open subset of the jet space then we call it an \emph{open relation}. A $C^r$ map $f:M\rightarrow N$ is said to be a \emph{solution} of $\mathcal{R}$ if the image of its $r$-jet extension $j^r_f:M\rightarrow J^r(M,N)$ lies in $\mathcal{R}$.

We denote by $\Gamma(\mathcal{R})$ the space of sections of the bundle $J^r(M,N)\rightarrow N$ having images in $\mathcal{R}$. The space of $C^\infty$ solutions of $\mathcal{R}$ is denoted by $Sol(\mathcal{R})$.
If $Sol(\mathcal R)$ and $\Gamma(\mathcal R)$ are endowed with the $C^{\infty}$-compact open topology and the $C^0$-compact open topology respectively, then the $r$-jet map \[j^r:Sol(\mathcal R)\to \Gamma(\mathcal R)\]
taking an $f\in Sol(\mathcal R)$ onto the holonomic section $j^r_f$ is a continuous map which is clearly one to one. Therefore, we can identify $Sol(\mathcal R)$ with the space of holonomic sections of $\mathcal R$.

\begin{definition}{\em A differential relation $\mathcal{R}$ is said to satisfy the \emph{$h$-principle} if every element $\sigma_{0} \in \Gamma(\mathcal{R})$ admits a homotopy $\sigma_{t}\in \Gamma(\mathcal{R})$ such that $\sigma_{1}$ is holonomic.

The relation $\mathcal R$ satisfies the \emph{parametric $h$-principle} \index{parametric $h$-principle} if the $r$-jet map $j^{r}:Sol(\mathcal{R})\rightarrow \Gamma(\mathcal{R})$ is a weak homotopy equivalence.
}\index{$h$-principle}\end{definition}
We shall often talk about (parametric) $h$-principle for certain function spaces without referring to the relations of which they are solutions.
\begin{remark}{\em The space $\Gamma(\mathcal R)$ is referred as the space of formal solutions of $\mathcal R$. Finding a formal solution is a purely (algebraic) topological problem which can be addressed with the obstruction theory. Finding a solution of $\mathcal R$ is, on the other hand, a differential topological problem. Thus, the $h$-principle reduces a differential topological problem to a problem in algebraic topology.}\end{remark}

Next we define the notion of local $h$-principle near a polyhedron.
\begin{definition}{\em Let $K$ be a subset of $M$. We shall say that a relation $\mathcal R$ satisfies the \emph{$h$-principle near $K$}  if given an open set $U$ containing $K$ and a section $F:U\to \mathcal{R}|_U$, there exists an open set $\tilde{U}\subset U$ containing $K$ such that $F|_{\tilde{U}}$ is  homotopic to a holonomic section $\tilde{F}:\tilde{U}\to \mathcal{R}$ in $\Gamma(\mathcal R)$.}\end{definition}
The above $h$-principle will also be referred as an $h$-principle on $Op\,K$.
If $K$ is a subset of $M$ then by $Op\,K$ we shall mean an unspecified open set in $M$ containing $K$. The set $C^k(Op\,K,N)$ will denote the set of all $C^k$ functions which are defined on some open neighbourhood of $K$.
\begin{definition}{\em A function $F:Z\to C^k({Op\,K},N)$ defined on any topological space $Z$ will be called `\emph{continuous}' if there exists an open set $U$ containing $K$ such that each $F(z)$ has an extension $\tilde{F}(z)$ which is defined on $U$ and $z\mapsto \tilde{F}(z)$ is continuous with respect to the $C^k$-compact open topology on the function space. A relation $\mathcal R$ is said to satisfy the \emph{parametric $h$-principle near} $K$ if $j^r:Sol(\mathcal R|_{Op\,K})\to \Gamma(\mathcal R|_{Op\,K})$ is a weak homotopy equivalence.}\end{definition}


Let $\text{Diff}(M)$ be the pseudogroup of local diffeomorphisms of $M$ \cite{geiges}. There is a natural (contravariant) action of $\text{Diff}(M)$ on $J^{r}(M,N)$ given by $\sigma .\alpha:=j^{r}_{f\circ \sigma}(x)$, where $\sigma$ is a local diffeomorphism of $M$ defined near $x\in M$ and $f$ is a representative of the $r$-jet $\alpha$ at $\sigma(x)$. Let $\mathcal D$ be a subgroup of $\text{Diff}(M)$. A differential relation $\mathcal{R}$ is said to be \emph{$\mathcal{D}$-invariant} if the following condition is satisfied:
\begin{center} For every $\alpha\in \mathcal R$ and $\sigma\in \mathcal D$, the element $\sigma.\alpha$ belongs to $\mathcal R$ provided it is defined.\end{center}
We shall denote the element $\sigma.\alpha$ by the notation $\sigma^*\alpha$

The following result, due to Gromov, is the first general result in the theory of $h$-principle.
\begin{theorem} Every open, Diff$(M)$ invariant relation $\mathcal R$ on an open manifold $M$ satisfies the parametric $h$-principle.\label{open-invariant}
\end{theorem}
The $h$-principle can be established in 2 steps. First one proves the local $h$-principle near the core $K$ of $M$ and then lifts the $h$-principle to $M$ by a contracting diffeotopy.

If a relation is invariant under the action of a smaller pseudogroup of diffeomorphism, say $\mathcal D$, then the $h$-principle can still hold if $\mathcal D$ has some additional properties.
\begin{definition}  {\em (\cite{gromov_pdr}) Let $M_{0}$ be a submanifold of $M$ of positive codimension and let $\mathcal{D}$ be a pseudogroup of local diffeomorphisms of $M$. We say that $M_0$ \emph{is sharply movable} by $\mathcal{D}$, if given any hypersurface $S$ in an open set $U$ in $M_0$ and any $\varepsilon>0$, there is an isotopy $\delta_{t}$, $t\in\I$, in $\mathcal{D}$ and a positive real number $r$ such that the following conditions hold:
\begin{enumerate}\item[$(i)$] $\delta_{0}|_U=id_{U}$,
\item[$(ii)$] $\delta_{t}$ fixes all points outside the $\varepsilon$-neighbourhood of $S$,
\item[$(iii)$] $dist(\delta_{1}(x),M_{0})\geq r$ for all $x\in S$ and for some $r>0$,\end{enumerate}
where $dist$ denotes the distance with respect to any fixed metric on $M$.}\label{D:sharp_diffeotopy}\index{sharp diffeotopy}\end{definition}
The diffeotopy $\delta_t$ will be referred as a \emph{sharply moving diffeotopy}.
A pseudogroup $\mathcal D$ is said to have the \emph{sharply moving property} if every submanifold $M_0$ of positive codimension is sharply movable by $\mathcal D$.

We end this section with the following result due to Gromov (\cite{gromov_pdr}).
\begin{theorem}
\label{T:gromov-invariant}
Let $\mathcal{R}\subset J^r(M,N)$ be an open relation which is invariant under the action of a pseudogroup $\mathcal{D}$. If $\mathcal{D}$ sharply moves a submanifold $M_{0}$ in $M$ of positive codimension then the parametric h-principle holds for $\mathcal{R}$ on $Op\,(M_{0})$.
\end{theorem}
We shall now prove Theorem~\ref{CT} by an application of the above theorem.\\
\emph{Proof of Theorem~\ref{CT}}.
Let $\mathcal{D}$ denote the pseudogroup of contact diffeomorphisms of $M$. We shall first show that $\mathcal{D}$ has the sharply moving property (see Definition~\ref{D:sharp_diffeotopy}). Let $M_0$ be a submanifold of $M$ of positive codimension. Take a closed hypersurface $S$ in $M_0$ and an open set $U\subset M$ containing $S$. We take a vector field $X$ along $S$ which is transversal to $M_{0}$. Let $H:M\rightarrow \mathbb{R}$ be a function such that
\[\alpha(X)=H,\ \ \ \  i_X d\alpha|_\xi=-dH|_\xi, \ \ \text{at points of } S.\]
(see equation~\ref{contact_hamiltonian1}). The contact-Hamiltonian vector field $X_H$ is clearly transversal to $M_0$ at points of $S$. As transversality is a stable property and $U$ is small, we can assume that $X_{H}\pitchfork U$. Now consider the initial value problem
\[\frac{d}{dt}\delta_{t}(x)=X_{H}(\delta_t(x)), \ \ \delta_0(x)=x\]
The solution to this problem exists for small time $t$, say for $t\in [0,\bar{\varepsilon}]$, for all $x$ lying in some small enough neighbourhood of $S$. Moreover, since $X_H$ is transversal to $S$, there would exist a positive real number $\vare$ such that the integral curves $\delta_t(x)$ for $x\in S$ do not meet $M_0$ during the time interval $(0,\vare)$. Let \[S_{\varepsilon}=\cup_{t\in[0,\varepsilon/2]}\delta_{t}(S).\]
Take a smooth function $\varphi$ which is identically equal to 1 on a small neighbourhood of $S_{\varepsilon}$ and supp\,$\varphi\subset \cup_{t\in[0,\varepsilon)}\delta_t(S)$. We then consider the initial value problem with $X_{H}$ replaced by $X_{\varphi H}$. Since $X_{\varphi H}$ is compactly supported the flow of $X_{\varphi H}$, say  $\bar{\delta_{t}}$, is defined for all time $t$. Because of the choice of $\varphi$, the integral curves $\bar{\delta}_t(x_0)$, $x_0\in M_0$, cannot come back to $M_0$ for $t>0$. Hence, we have the following:
\begin{itemize}
 \item $\bar{\delta}_0|_U=id_U$
 \item $\bar{\delta_t}=id$ outside a small neighbourhood of $S_{\varepsilon}$
 \item $dist(\bar{\delta}_1(x),M_0)>r$ for all $x\in S$ and for some $r>0$.
\end{itemize}
This proves that $\mathcal{D}$ sharply moves any submanifold of $M$ of positive codimension.

Since $M$ is open it has a core $K$ which is of positive codimension. Since the relation $\mathcal R$ is open and invariant under the action of $\mathcal D$, we can apply Theorem~\ref{T:gromov-invariant} to conclude that $\mathcal R$ satisfies the parametric $h$-principle near $K$. We shall now lift the $h$-principle from $Op\,K$ to all of $M$ by appealing to Corollary~\ref{CO}.

By the local $h$-principle near $K$, an arbitrary section $F_0$ of $\mathcal R$ admits a homotopy $F_{t}$ in $\Gamma(\mathcal{R}|_U)$ such that $F_1$ is
holonomic on $U$, where  $U$ is an open neighbourhood of $K$ in $M$.
Let $f_{t}=p^{(r)}\circ F_t$, where $p^{(r)}:J^{r}(M,N)\rightarrow N$ is the canonical projection map of the jet bundle. By Corollary~\ref{CO} above we get a homotopy of isocontact immersions $g_{t}:(M,\xi)\rightarrow (M,\xi)$ satisfying $g_{0}=id_{M}$ and $g_{1}(M)\subset U$, where $\xi=\ker\alpha$. The concatenation of the homotopies $g_t^*(F_0)$ and $g_1^*(F_t)$ gives the desired homotopy in $\Gamma(\mathcal R)$ between $F_0$ and the holonomic section $g_1^*(F_1)$. This proves that $\mathcal R$ satisfies the ordinary $h$-principle.

To prove the parametric $h$-principle, take a parametrized section $F_z\in \Gamma(\mathcal{R})$, $z\in \D^n$, such that $F_z$ is holonomic for all $z\in \mathbb S^{n-1}$. This implies that there is a family of smooth maps $f_z\in Sol(\mathcal R)$, parametrized by $z\in \mathbb S^{n-1}$, such that $F_z=j^r_f(z)$. We shall  homotope the parametrized family $F_z$ to a family of holonomic sections in $\mathcal R$ such that the homotopy remains constant on $\mathbb S^{n-1}$. By the parametric $h$-principle near $K$, there exists an open neighbourhood $U$ of $K$ and a homotopy $H:\D^n\times\I\to \Gamma(\mathcal R|_U)$, such that $H^0_z=F_z$ and $H_z^1$ is holonomic for all $z\in \D^n$; furthermore,  $H_z^t=j^r_f(z)$ on $U$ for all $z\in \mathbb S^{n-1}$.

Let $\delta:[0,1/2]\to [0,1]$ be the linear homeomorphism such that $\delta(0)=0$ and $\delta(1/2)=1$. Define a function $\mu$ as follows:
\begin{eqnarray*}\mu(z)   = &  \delta(\|z\|)z/\|z\| & \text{ if }\ \|z\|\leq 1/2.\end{eqnarray*}
 First deform $F_z$ to $\widetilde{F}_z$, where
\[\begin{array}{rcl}\widetilde{F}_z & = & \left\{
\begin{array}{ll}
F_{\mu(z)} & \text{if } \|z\|\leq 1/2\\
F_{z/\|z\|} & \text{if } 1/2\leq\|z\|\leq 1\end{array}\right.\end{array}\]
Let $\bar{\delta}:[1/2,1]\to [0,1]$ be the linear homeomorphism such that $\bar{\delta}(1/2)=1$ and $\bar{\delta}(1)=0$.
Define a homotopy $\widetilde{F}^s_z$ of $\tilde{F}_z$ as follows:
\[\begin{array}{rcl}\widetilde{F}_z^s & = & \left\{
\begin{array}{ll}
g_s^*(F_{\mu(z)}), &  \|z\|\leq 1/2\\
g^*_{s\bar{\delta}(\|z\|)}(F_{z/\|z\|}) & 1/2\leq\|z\|\leq 1\end{array}\right.\end{array}\]
Note that
\[\begin{array}{rcl}\widetilde{F}^1_z & = & \left\{
\begin{array}{ll}
g_1^*(F_{\mu(z)}), & \|z\|\leq 1/2\\
g^*_{\bar{\delta}(\|z\|)}(F_{z/\|z\|}) & 1/2\leq\|z\|\leq 1\end{array}\right.\end{array}\]
Finally we consider a parametrized homotopy given as follows:
\[\begin{array}{rcl}\widetilde{H}^s_z & = & \left\{
\begin{array}{ll}
g_1^*(H^s_{\mu(z)}), & \|z\|\leq 1/2\\
g^*_{\bar{\delta}(\|z\|)}(F_{z/\|z\|}) & 1/2\leq\|z\|\leq 1\end{array}\right.\end{array}\]
Note that $\widetilde{H}^1_z$ is holonomic for all $z\in\D^n$ and $\widetilde{H}^s_z=j^r_f(z)$ for all $z\in \mathbb S^{n-1}$. The concatenation of the three homotopies now give a homotopy between the parametrized sections $F_z$ and $\tilde{H}^1_z$ relative to $\mathbb S^{n-1}$. This proves the parametric $h$-principle for $\mathcal R$.
\qed\\

\section{Transversality Theorem on open contact manifolds}
Throughout this section, $M$ is a contact manifold with a given contact form $\alpha$ and $N$ is a foliated manifold with a smooth foliation $\mathcal F_N$ of even codimension.
\begin{definition} {\em  A foliation $\mathcal F$ on $M$ will be called a \emph{contact foliation subordinate to} $\alpha$ or, a \emph{contact foliation on} $(M,\alpha)$ if the leaves of $\mathcal F$ are contact submanifolds of $(M,\alpha)$.\label{subordinate_contact_foliation}}
\end{definition}
Recall that a leaf $L$ of an arbitrary foliation on $M$ admits an injective immersion $i_L:L\to M$. We shall say that $L$ is a contact submanifold of $(M,\alpha)$ if the pullback form $i_L^*\alpha$ is a contact form on $L$.
\begin{remark}{\em In view of Lemma~\ref{L:contact_submanifold}, $\mathcal F$ is a contact foliation on $(M,\alpha)$ if and only if $T\mathcal F$ is transversal to the contact distribution $\ker\alpha$ and $T\mathcal F\cap \ker\alpha$ is a symplectic subbundle of $(\ker\alpha,d'\alpha)$.}\label{R:tangent_contact_foliation}\end{remark}
Let $Tr_\alpha(M,\mathcal F_N)$ and $\mathcal E_\alpha(TM,\nu\mathcal F_N)$ be as in Section 1.
We define a first order differential relation $\mathcal R$ consisting of all 1-jets represented by triples $(x,y,F)$, where $x\in M, y\in N$ and $F:T_{x}M\rightarrow T_{y}N$ is a linear map such that
\begin{enumerate}\item $\pi\circ F:T_xM\to \nu(\mathcal F_N)_y$ is an epimorphism
\item $\ker(\pi\circ F)\cap \ker\alpha_x$ is a symplectic subspace of $(\ker\alpha_x,d'\alpha_x)$. \end{enumerate}
Then it is easy to note that the space of sections of $\mathcal R$ can be identified with $\mathcal E_\alpha(TM,\nu(\mathcal F_N))$.

\begin{observation} {\em The solution space of $\mathcal R$ is the same as $Tr_\alpha(M,\mathcal F)$. To see this, it is sufficient to note (see Definition~\ref{contact_submanifold}) that the following two statements are equivalent:
\begin{enumerate}
\item[(S1)] $f:M\to N$ is transversal to $\mathcal F_N$ and the leaves of the inverse foliation $f^*\mathcal F_N$ are contact submanifolds (immersed) of $M$.
\item[(S2)] $\pi\circ df$ is an epimorphism and $\ker (\pi\circ df)\cap \ker\alpha$ is a symplectic subbundle of $(\ker\alpha,d'\alpha)$.
\end{enumerate}}
Hence Theorem~\ref{T:contact-transverse} states that the relation $\mathcal R$ satisfies the parametric $h$-principle.
\label{P:solution space}\end{observation}
We will now show that the relation $\mathcal R$ is open and invariant under the action of local contactomorphisms.

\begin{lemma}
\label{OR}
 The relation $\mathcal{R}$ defined above is an open relation.
\end{lemma}
\begin{proof} Let $V$ be a $(2m+1)$-dimensional vector space with a (linear) 1-form $\theta$ and a 2-form $\tau$ on it such that $\theta \wedge \tau^{m}\neq 0$. We shall call $(\theta,\tau)$ an almost contact structure on $V$. Note that the restriction of $\tau$ to $\ker\theta$ is then non-degenerate. A subspace $K$ of $V$ will be called an almost contact subspace if the restrictions of $\theta$ and $\tau$ to $K$ define an almost contact structure on $K$. In this case, $K$ must be transversal to $\ker\theta$ and $K\cap \ker\theta$ will be a symplectic subspace of $\ker\theta$.

Let $W$ be a vector space of even dimension and $Z$ a subspace of $W$ of codimension $2q$. Denote by $L_Z^\pitchfork(V,W)$ the set of all linear maps $L:V\to W$ which are transversal to $Z$. This is clearly an open subset in the space of all linear maps from $V$ to $W$. Define a subset $\mathcal L$ of $L_Z^\pitchfork(V,W)$ by
\[\mathcal L=\{L\in L_Z^\pitchfork(V,W)| \ker(\pi\circ L) \text{ is an almost contact subspace of }V\}\]
We shall prove that $\mathcal L$ is an open subset of $L_Z^\pitchfork(V,W)$.
Consider the map
\[E:L_Z^\pitchfork(V,W)\rightarrow Gr_{2(m-q)+1}(V)\]
\[L\mapsto \ker (\pi\circ L),\]
where $\pi:W\to W/Z$ is the quotient map.
Let $\mathcal U_c$ denote the subset of $G_{2(m-q)+1}(V)$ consisting of all almost contact subspaces $K$ of $V$. Observe that $\mathcal L=E^{-1}(\mathcal U_c)$. We shall now prove that
\begin{itemize}\item $E$ is a continuous map and
\item $\mathcal U_c$ is an open subset of $G_{2(m-q)+1}(V)$.
\end{itemize}
To prove that $E$ is continuous, take $L_0\in L_Z^\pitchfork(V,W)$ and let $K_0=\ker (\pi\circ L_0)$. Consider the subbasic open set $U_{K_0}$ consisting of all subspaces $Y$ of $V$ such that the canonical projection $p:K_0\oplus K_0^\perp\to K_0$ maps $Y$ isomorphically onto $K_0$. The inverse image of $U_{K_0}$ under $E$ consists of all $L:V\to W$ such that $p|_{\ker (\pi\circ L)}:\ker (\pi\circ L)\to K_0$ is onto.
It may be seen easily that if $L\in L_Z^\pitchfork(V,W)$ then
\begin{eqnarray*}  p \text{ maps } \ker (\pi\circ L) \text{ onto }K_0 &
  \Leftrightarrow &  \ker (\pi\circ L)\cap K_0^\perp=\{0\} \\
 &  \Leftrightarrow &  \pi\circ L|_{K_0^\perp}:K_0^\perp\to W/Z \text{ is an isomorphism}.\end{eqnarray*}
Now, the set of all $L$ such that $\pi\circ L|_{K_0^\perp}$ is an isomorphism is an open subset. Hence $E^{-1}(U_{K_0})$ is open and therefore $E$ is continuous.

To prove the openness of $\mathcal U_c$ take $K_0\in\mathcal U$. Recall that a subbasic open set $U_{K_0}$ containing $K_0$ can be identified with the space $L(K_0,K_0^\perp)$,  where $K_0^\perp$ denotes the orthogonal complement of $K$ with respect to some inner product on $V$ (\cite{milnor_stasheff}).
Let $\Theta$ denote the following composition of continuous maps:
\[\begin{array}{rcccl}U_{K_0}\cong L(K_0,K_0^{\perp}) & \stackrel{\Phi}{\longrightarrow} &  L(K_0,V)& \stackrel{\Psi}{\longrightarrow} & \Lambda^{2(m-q)+1}(K_0^*)\cong\R\end{array}\]
where $\Phi(L)=I+L$ and  $\Psi(L)=L^*(\theta\wedge\tau^{2(m-q)})$.
It may be noted that, if $K\in U_{K_0}$ is mapped onto some $T\in L(K_0,V)$ then the image of $T$ is $K$. Hence it follows that
\[{\mathcal U}_c\cap U_{K_0}=(\Psi\circ\Phi)^{-1}(\R\setminus 0)\]
which proves that ${\mathcal U}_c\cap U_{K_0}$ is open. Since $U_{K_0}$ is a subbasic open set in the topology of Grassmannian it proves the openness of $\mathcal U_c$.
Thus $\mathcal L$ is an open subset.

We now show that $\mathcal R$ is an open relation. First note that, each tangent space $T_xM$ has an almost contact structure given by $(\alpha_x,d\alpha_x)$.
Let $U$ be a trivializing neighbourhood of the tangent bundle $TM$. We can choose a trivializing neighbourhood $\tilde{U}$ for the tangent bundle $TN$ such that $T\mathcal F_N$ is isomorphic to $\tilde{U}\times Z$ for some codimension $2q$-vector space in $\R^{2n}$. This implies that $\mathcal R\cap J^1(U,\tilde{U})$ is diffeomorphic with  $U\times\tilde{U}\times\mathcal L$. Since the sets $J^1(U,\tilde{U})$ form a basis for the topology of the jet space, this completes the proof of the lemma.
\end{proof}

\begin{lemma}
\label{IV}
$\mathcal{R}$ is invariant under the action of the pseudogroup of local contactomorphisms of $(M,\alpha)$.
\end{lemma}
\begin{proof}
Let $\delta$ be a local diffeomorphism on an open neighbourhood of $x\in M$ such that $\delta^*\alpha=\lambda\alpha$, where $\lambda$ is a nowhere vanishing function on $Op\, x$. This implies that $d\delta_x(\xi_x)=\xi_{\delta(x)}$ and $d\delta_x$ preserves the conformal symplectic structure determined by $d\alpha$ on $\ker \xi$. If $f$ is a local solution of $\mathcal R$ at $\delta(x)$, then
\[d\delta_x(\ker d(f\circ\delta)_x\cap \xi_x)=\ker df_{\delta(x)}\cap\xi_{\delta(x)}.\]
Hence $f\circ\delta$ is also a local solution of $\mathcal R$ at $x$.
Since $\mathcal R$ is open every representative function of a jet in $\mathcal R$ is a local solution of $\mathcal R$. Thus
local contactomorphisms act on $\mathcal R$ by $\delta.j^1_f(\delta(x)) = j^1_{f\circ\delta}(x)$.
\end{proof}

\emph{Proof of Theorem~\ref{T:contact-transverse}}:
In view of Theorem~\ref{CT}, and Lemma~\ref{OR}, \ref{IV} it follows that the relation $\mathcal R$ satisfies the parametric $h$-principle. This completes the proof by Observation ~\ref{P:solution space}.\qed\\
\begin{definition} \emph{A smooth submersion $f:(M,\alpha)\to N$ is called a \emph{contact submersion} if the level sets of $f$ are contact submanifolds of $M$.}
\end{definition}
We shall denote the space of contact submersion $(M,\alpha)\to N$ by $\mathcal C_\alpha(M,N)$.
The space of epimorphisms $F:TM\to TN$ for which $\ker F\cap \ker\alpha$ is a symplectic subbundle of $(\ker\alpha,d'\alpha)$ will be denoted by $\mathcal E_\alpha(TM,TN)$.
Taking $\mathcal F_N$ to be the zero-dimensional foliation on $N$ in Theorem~\ref{T:contact-transverse} we get the following result.
\begin{corollary} Let $(M,\alpha)$ be an open contact manifold. The derivative map
 \[d:\mathcal C_\alpha(M,N)\to \mathcal E_\alpha(TM,TN)\]
is a weak homotopy equivalence.\label{T:contact_submersion}
\end{corollary}
\begin{remark}{\em Suppose that $F_0\in \mathcal E_\alpha(TM,TN)$ and $D$ is the kernel of $F_0$. Then $(D,\alpha|_D,d\alpha|_D)$ is an almost contact distribution. Since $M$ is an open manifold, the bundle epimorphism $F_0:TM\to TN$ can be homotoped (in the space of bundle epimorphism) to the derivative of a submersion $f:M\to N$ (\cite{phillips}). Hence the distribution $\ker F_0$ is homotopic to an integrable distribution, namely the one given by the submersion $f$. It then follows from a result proved in \cite{datta_mukherjee} that $(D,\alpha|_D,d\alpha|_D)$ is homotopic to the distribution associated to a contact foliation $\mathcal F$ on $M$. Theorem~\ref{T:contact-transverse} further implies that it is possible to get a foliation $\mathcal F$ which is subordinate to $\alpha$ and is defined by a submersion.}\end{remark}

\section{Foliations and $\Gamma_q$-structures}
\subsection{$\Gamma$-structures \label{classifying space}} We first review some basic facts about $\Gamma$-structures for a topological groupoid $\Gamma$ following \cite{haefliger}. We also recall the connection between foliations on manifolds and $\Gamma_q$ structures, where $\Gamma_q$ is the groupoid of germs of local diffeomorphisms of $\R^q$\index{$\Gamma_q$}). For preliminaries of topological groupoid we refer to \cite{moerdijk}.
\begin{definition}\label{GS}{\em
Let $X$ be a topological space with an open covering $\mathcal{U}=\{U_i\}_{i\in I}$ and let $\Gamma$ be a topological groupoid over a space $B$. A 1-cocycle on $X$ over $\mathcal U$ with values in $\Gamma$ is a collection of continuous maps \[\gamma_{ij}:U_i \cap U_j\rightarrow \Gamma\] such that \[\gamma_{ik}(x)=\gamma_{ij}(x)\gamma_{jk}(x),\ \text{ for all }\ x\in U_i \cap U_j \cap U_k. \]
The above conditions imply that $\gamma_{ii}$ has its image in the space of units of $\Gamma$ which can be identified with $B$ via the unit map $1:B\to \Gamma$. We call two 1-cocycles $(\{U_i\}_{i\in \I},\gamma_{ij})$ and $(\{\tilde{U}_k\}_{k\in K},\tilde{\gamma}_{kl})$ equivalent if for each $i\in I$ and $k\in K$, there are continuous maps \[\delta_{ik}:U_i \cap \tilde{U}_k\rightarrow \Gamma\] such that
\[\delta_{ik}(x)\tilde{\gamma}_{kl}(x)=\delta_{il}(x)\ \text{for}\ x\in U_i \cap \tilde{U}_k \cap \tilde{U}_l\]
\[\gamma_{ji}(x)\delta_{ik}(x)=\delta_{ij}(x)\ \text{for}\ x\in U_i \cap U_j \cap \tilde{U}_k.\]
An equivalence class of a 1-cocycle is called a $\Gamma$-\emph{structure}\index{$\Gamma$-structure}. These structures have also been referred as Haefliger structures in the later literature.}
\end{definition}
For a continuous map $f:Y\rightarrow X$ and a $\Gamma$-structure $\Sigma=(\{U_i\}_{i\in I},\gamma_{ij})$ on $X$, the \emph{pullback $\Gamma$-structure} $f^*\Sigma$ is defined by the covering $\{f^{-1}U_i\}_{i \in I}$ together with the cocycles $\gamma_{ij}\circ f$.

If $f,g:Y\to X$ are homotopic maps and $\Sigma$ is a $\Gamma$-structure on $X$ then the pull-back structures $f^*\Sigma$ and $g^*\Sigma$ are not the same. They are homotopic in the following sense.
\begin{definition}{\em
Two $\Gamma$-structures $\Sigma_0$ and $\Sigma_1$ on a topological space $X$ are called \emph{homotopic} if there exists a $\Gamma$-structure $\Sigma$ on $X\times I$, such that $i_0^*\Sigma=\Sigma_0$ and $i_1^*\Sigma=\Sigma_1$, where $i_0:X\to X\times I$ and $i_1:X\to  X\times I$ are canonical injections defined by $i_t(x)=(x,t)$ for $t=0,1$.}\end{definition}

\begin{definition}{\em
Let $\Gamma$ be a topological groupoid with space of units $B$, source map $\mathbf{s}$ and target map $\mathbf{t}$. Consider the infinite sequences \[(t_0,x_0,t_1,x_1,...)\] with $t_i \in [0,1],\ x_i \in \Gamma$ such that all but finitely many $t_i$'s are zero and $\mathbf{t}(x_i)=\mathbf{t}(x_j)$ for all $i,j$. Two such sequences \[(t_0,x_0,t_1,x_1,...)\] and \[(t'_0,x'_0,t'_1,x'_1,...)\] are called equivalent if $t_i=t'_i$ for all $i$ and $x_i=x'_i$ for all $i$ with  $t_i\neq 0$. Denote the set of all equivalence classes by $E\Gamma$. The topology on $E\Gamma$ is defined to be the weakest topology such that the following set maps are continuous:
\[t_i:E\Gamma \rightarrow [0,1]\ \text{ given by }\ (t_0,x_0,t_1,x_1,...)\mapsto t_i \]
\[x_i: t_i^{-1}(0,1] \rightarrow \Gamma \ \text{ given by }\ (t_0,x_0,t_1,x_1,...)\mapsto x_i.\]
There is also a `$\Gamma$-action' on $E\Gamma$ as follows: Two elements $(t_0,x_0,t_1,x_1,...)$ and $(t'_0,x'_0,t'_1,x'_1,...)$ in $E\Gamma$ are said to be $\Gamma$-equivalent if $t_i=t'_i$ for all $i$,  and  if there exists a $\gamma\in \Gamma$ such that $x_i=\gamma x'_i$ for all $i$ with $t_i\neq 0$. The set of equivalence classes with quotient topology is called the \emph{classifying space of} $\Gamma$, and is denoted by $B\Gamma$\index{$B\Gamma$}.}
\end{definition}

Let $p: E\Gamma \rightarrow B\Gamma$ denote the quotient map. The maps $t_i:E\Gamma \rightarrow [0,1]$ project down to maps $u_i:B\Gamma \rightarrow [0,1]$ such that $u_i \circ p=t_i$. The classifying space $B\Gamma$ has a natural $\Gamma$-structure $\Omega=(\{V_i\}_{i\in I},\gamma_{ij})$, where $V_i=u_i^{-1}(0,1]$ and $\gamma_{ij}:V_i \cap V_j \rightarrow \Gamma$ is given by
\[(t_0,x_0,t_1,x_1,...)\mapsto x_i x_j^{-1}\]
We shall refer to this $\Gamma$ structure as the \emph{universal $\Gamma$-structure}\index{universal $\Gamma$-structure}.

For any two topological groupoids $\Gamma_1,\Gamma_2$ and for a groupoid homomorphism $f:\Gamma_1\rightarrow \Gamma_2$ there exists a continuous map \[Bf:B\Gamma_1\rightarrow B\Gamma_2,\]
defined by the functorial construction.
\begin{definition}{\em (Numerable $\Gamma$-structure)
Let $X$ be a topological space. An open covering $\mathcal{U}=\{U_i\}_{i\in I}$ of $X$ is called \emph{numerable} if it admits a locally finite partition of unity $\{u_i\}_{i\in I}$, such that $u_i^{-1}(0,1]\subset U_i$. If a $\Gamma$-structure can be represented by a 1-cocycle whose covering is numerable then the $\Gamma$-structure is called \emph{numerable}. }
\end{definition}
It can be shown that every $\Gamma$-structure on a paracompact space is numerable.
\begin{definition}{\em Let $X$ be a topological space. Two numerable $\Gamma$-structures are called \emph{numerably homotopic} if there exists a homotopy of numerable $\Gamma$-structures joining them.}
\end{definition}
Haefliger proved that the homotopy classes of numerable $\Gamma$-structures on a topological space $X$ are in one-to-one correspondence with the homotopy classes of continuous maps $X\to B\Gamma$.
\begin{theorem}(\cite{haefliger1})
\label{CMT} Let $\Gamma$ be a topological groupoid and $\Omega$ be the universal $\Gamma$ structure on $B\Gamma$. Then
\begin{enumerate}
\item $\Omega$ is numerable.
\item If $\Sigma$ is a numerable $\Gamma$-structure on a topological space $X$, then there exists a continuous map $f:X\rightarrow B\Gamma$ such that $f^*\Omega$ is homotopic to $\Sigma$.
\item If $f_0,f_1:X\rightarrow B\Gamma$ are two continuous functions, then $f_0^*\Omega$ is numerably homotopic to $f_1^*\Omega$ if and only if $f_0$ is homotopic to $f_1$.
\end{enumerate}
\end{theorem}
\subsection{$\Gamma_q$-structures and their normal bundles}
We now specialise to the groupoid $\Gamma_q$ of germs of local diffeomorphisms of $\mathbb{R}^{q}$. The source map $\mathbf s:\Gamma_q\to \R^q$ and the target map $\mathbf t:\Gamma_q\to \R^q$ are defined as follows: If $\phi\in\Gamma_q$ represents a germ at $x$, then
\[{\mathbf s}(\phi)=x\ \ \text{ and }\ \ {\mathbf t}(\phi)=\phi(x)\]
The units of $\Gamma_q$ consists of the germs of the identity map at points of $\R^q$.
$\Gamma_q$ is topologised as follows: For a local diffeomorphism $f:U\rightarrow f(U)$, where $U$ is an open set in $\mathbb{R}^q$, define $U(f)$ as the set of germs of $f$ at different points of $U$. The collection of all such $U(f)$ forms a basis of some topology on $\Gamma_q$ which makes it a topological groupoid. The derivative map gives a groupoid homomorphism \[\bar{d}:\Gamma_q \rightarrow GL_q(\mathbb{R})\]
which takes the germ of a local diffeomorphism $\phi$ of $\R^q$ at $x$ onto $d\phi_x$.
Thus, to each $\Gamma_q$-structure $\omega$ on a topological space $M$ there is an associated (isomorphism class of) $q$-dimensional vector bundle $\nu(\omega)$ over $M$ which is called the \emph{normal bundle of} $\omega$. In fact, if $\omega$ is defined by the cocycles $\gamma_{ij}$ then the cocycles $\bar{d}\circ \gamma_{ij}$ define the vector bundle  $\nu(\omega)$. Moreover, two equivalent cocycles in $\Gamma_q$ have their normal bundles isomorphic. Thus the normal bundle of a $\Gamma_q$ structure is the isomorphism class of the normal bundle of any representative cocycle. If two $\Gamma_q$ structures $\Sigma_0$ and $\Sigma_1$ are homotopic then there exists a $\Gamma_q$ structure $\Sigma$ on $X\times I$ such that $i_0^*\Sigma=\Sigma_0$ and $i_1^*\Sigma=\Sigma_1$, where $i_0:X\to X\times \{0\}\hookrightarrow X\times I$ and $i_1:X\to X\times \{1\}\hookrightarrow X\times I$ are canonical injective maps. Then $\nu(i_0^*\Sigma_0)\cong i_0^*\nu(\Sigma)\cong i_1^*\nu(\Sigma)\cong \nu(i_1^*\Sigma_1)$. Hence, normal bundles of homotopic $\Gamma_q$ structures are isomorphic.

In particular, we have a vector bundle $\nu\Omega_q$ on $B\Gamma_q$ associated with the universal $\Gamma_q$-structure $\Omega_q$ \index{$\Omega_q$} on $B\Gamma_q$.
\begin{proposition}If a continuous map $f:X\to B\Gamma_q$ classifies a $\Gamma_q$-structure $\omega$ on a topological space $X$, then $Bd\circ f$ classifies the vector bundle $\nu(\omega)$. In particular, $\nu\Omega_q\cong Bd^*E(GL_q(\R))$ and hence $\nu(\omega)\cong f^*\nu\Omega_q$.
\end{proposition}

\subsection{$\Gamma_q$-structures vs. foliations}

If a foliation  $\mathcal F$ on a manifold $M$ is represented by the Haefliger data $\{U_i,s_i,h_{ij}\}$, then we can define a $\Gamma_q$ structure on $M$ by $\{U_i,g_{ij}\}$, where \[g_{ij}(x) = \text{ the germ of } h_{ij} \text{ at } s_i(x) \text{ for }x\in U_i\cap U_j.\]
In particular, $g_{ii}(x)$ is the germ of the identity map of $\R^q$ at $s_i(x)$ and hence $g_{ii}$ takes values in the units of $\Gamma_q$. If we identify the units of $\Gamma_q$ with $\R^q$, then $g_{ii}$ may be identified with $s_i$ for all $i$. Thus, one arrives at a $\Gamma_q$-structure $\omega_{\mathcal F}$ represented by 1-cocycles $(U_i,g_{ij})$ such that \[g_{ii}:U_i\rightarrow \mathbb{R}^q\subset \Gamma_q\] are submersions for all $i$. The functions $\tau_{ij}:U_i\cap U_j\to GL_q(\R)$ defined by $\tau_{ij}(x)=(\bar{d}\circ g_{ij})(x)$ for $x\in U_i\cap U_j$, define the normal bundle of $\omega_{\mathcal F}$. Furthermore, since $\tau_{ij}(x)=dh_{ij}(s_i(x))$, $\nu(\omega_{\mathcal F})$ is isomorphic to the quotient bundle $\nu(\mathcal F)$. Thus a foliation on a manifold $M$ defines a $\Gamma_q$-structure whose normal bundle is embedded in $TM$.

As we have noted above, foliations do not behave well under the pullback operation, unless the maps are transversal to foliations. However, in view of the relation between foliations and $\Gamma_q$ structures, it follows that the inverse image of a foliation by any map gives a $\Gamma_q$-structure. The following result due to Haefliger says that any $\Gamma_q$ structure is of this type.
\begin{theorem}(\cite{haefliger1})
\label{HL}
Let $\Sigma$ be a $\Gamma_{q}$-structure on a manifold $M$. Then there exists a manifold $N$, a closed embedding $s:M \hookrightarrow N$ and a $\Gamma_{q}$-foliation $\mathcal{F}_N$ on $N$ such that $s^*(\mathcal{F}_N)=\Sigma$ and $s$ is a cofibration.
\end{theorem}
Using the above theorem and the transversality result due to Phillips \cite{phillips1}, Haefliger gave the following classification of foliations on open manifolds.
\begin{theorem} The integrable homotopy classes of foliations on an open manifolds are in one-to-one correspondence with the homotopy classes of epimorphisms $TM\to \nu\Omega$.
\end{theorem}

\section{Classification of contact foliations}

Throughout this section $M$ is a contact manifold with a contact form $\alpha$. As before $\xi$ will denote the associated contact structure $\ker\alpha$ and $d'\alpha=d\alpha|_{\xi}$. Let $Fol_\alpha^{2q}(M)$ denote the space of contact foliations on $M$ of codimension $2q$ subordinate to $\alpha$ (Definition~\ref{subordinate_contact_foliation}).
Let $\mathcal E_{\alpha}(TM,\nu\Omega_{2q})$ be the space of all vector bundle epimorphisms $F:TM\to \nu \Omega_{2q}$ such that $\ker F$ is transversal to $\ker\alpha$ and $\ker\alpha\cap \ker F$ is a symplectic subbundle of $(\ker\alpha,d'\alpha)$.

If $\mathcal F\in Fol^{2q}(M)$ and $f:M\to B\Gamma_{2q}$ is a classifying map of $\mathcal F$, then  $f^*\Omega_{2q}= \mathcal F$ as $\Gamma_{2q}$-structure. We can define a vector bundle epimorphisms $TM\to \nu\Omega_{2q}$ by the following diagram (see \cite{haefliger1})
\begin{equation}
 \xymatrix@=2pc@R=2pc{
TM \ar@{->}[r]^-{\pi_{\mathcal F}}\ar@{->}[rd] & \nu \mathcal{F}\cong f^*(\nu \Omega_{2q}) \ar@{->}[r]^-{\bar{f}}\ar@{->}[d] & \nu \Omega_{2q} \ar@{->}[d]\\
& M \ar@{->}[r]_-{f} & B\Gamma_{2q}
}\label{F:H(foliation)}
\end{equation}
where $\pi_{\mathcal F}:TM\to \nu(\mathcal F)$ is the projection map onto the normal bundle and $(\bar{f},f)$ represents a pull-back diagram. Note that the kernel of this morphism is $T\mathcal F$ and therefore,
if $\mathcal F\in Fol^{2q}_\alpha(M)$, then  $\bar{f}\circ \pi_{\mathcal F} \in \mathcal E_\alpha(TM,\nu\Omega_{2q})$ (see Remark~\ref{R:tangent_contact_foliation}).
However, the morphism $\bar{f}\circ \pi_{\mathcal F}$ is defined uniquely only up to homotopy. Thus, there is a function
\[H'_\alpha:Fol^{2q}_\alpha(M)\to \pi_0(\mathcal E_\alpha(TM,\nu\Omega_{2q})).\]
\begin{definition} {\em Two contact foliations $\mathcal F_0$ and $\mathcal F_1$ on $(M,\alpha)$ are said to be \emph{integrably homotopic relative to $\alpha$} if there exists a foliation $\tilde{\mathcal F}$ on $(M\times\I,\alpha\oplus 0)$ such that the following conditions are satisfied:
\begin{enumerate}
\item $\tilde{\mathcal F}$ is transversal to the trivial foliation of $M\times\I$ by the leaves $M\times\{t\}$, $t\in \I$;
\item the foliation $\mathcal F_t$ on $M$ induced by the canonical injective map $i_t:M\to M\times\I$ (given by $x\mapsto (x,t)$) is a contact foliation subordinate to $\alpha$ for each $t\in\I$;
\item the induced foliations on  $M\times\{0\}$ and $M\times\{1\}$ coincide with $\mathcal F_0$ and $\mathcal F_1$ respectively,\end{enumerate}
Here $\alpha\oplus 0$ denotes the pull-back of $\alpha$ by the projection map $p_1:M\times\R\to M$.}\end{definition}

Let $\pi_0(Fol^{2q}_{\alpha}(M))$ denote the space of integrable homotopy classes of contact foliations on $(M,\alpha)$. Define
\[H_\alpha:\pi_0(Fol^{2q}_\alpha(M))\to \pi_0(\mathcal E_\alpha(TM,\nu\Omega_{2q}))\]
by $H_{\alpha}([\mathcal{F}])=H_\alpha'(\mathcal F)$, where $[\mathcal F]$ denotes the integrable homotopy class of $\mathcal F$ relative to $\alpha$.
To see that $H_\alpha$ is well-defined, let $\tilde{\mathcal F}$ be an integrable homotopy relative to $\alpha$ between two contact foliations $\mathcal F_0$ and $\mathcal F_1$.
If $F:M\times[0,1]\to B\Gamma_{2q}$ is a classifying map of $\tilde{\mathcal F}$ then we have a diagram similar to (\ref{F:H(foliation)}) given as follows:
\[
 \xymatrix@=2pc@R=2pc{
T(M\times[0,1]) \ar@{->}[r]^-{\bar{\pi}}\ar@{->}[rd] & \nu \tilde{F} \ar@{->}[r]^-{\bar{F}}\ar@{->}[d] & \nu \Omega_{2q} \ar@{->}[d]\\
& M\times [0,1] \ar@{->}[r]_-{F} & B\Gamma_{2q}
}
\]
Let $i_t:M\to M\times\{t\}\hookrightarrow M\times\R$ denote the canonical injective map of $M$ into $M\times\{t\}$ and $f_t:M\to B\Gamma_{2q}$ be defined as $f_t(x)=F(x,t)$ for $(x,t)\in M\times[0,1]$. Since $\bar{F}\circ \bar{\pi}\circ di_t:TM\to \nu(\Omega_{2q})$ represents the homotopy class $H'_\alpha(f_t^*\Omega_{2q})$ we conclude that  $H'_\alpha(\mathcal F_0)= H'_\alpha(\mathcal F_1)$.
This proves that $H_\alpha$ is well-defined. We now state the main result of this article.
\begin{theorem}
\label{haefliger_contact}
If $M$ is open then $H_\alpha:\pi_0(Fol^{2q}_{\alpha}(M)) \longrightarrow \pi_0(\mathcal E_{\alpha}(TM,\nu\Omega_{2q}))$ is bijective.
\end{theorem}
We first prove a lemma.
\begin{lemma}Let $N$ be a smooth manifold with a foliation $\mathcal F_N$ of codimension $2q$. If $g:N\to B\Gamma_{2q}$ classifies $\mathcal F_N$ then we have the commutative diagram as follows:
\begin{equation}
 \xymatrix@=2pc@R=2pc{
\pi_0(Tr_{\alpha}(M,\mathcal{F}_N))\ar@{->}[r]^-{P}\ar@{->}[d]_-{\cong}^-{\pi_0(\pi \circ d)} & \pi_0(Fol^{2q}_{\alpha}(M))\ar@{->}[d]^-{H_{\alpha}}\\
\pi_0(\mathcal E_{\alpha}(TM,\nu \mathcal{F}_N))\ar@{->}[r]_{G_*} & \pi_0(\mathcal E_{\alpha}(TM,\nu\Omega_{2q}))
}\label{Figure:Haefliger}
\end{equation}
where the left vertical arrow is the isomorphism defined by Theorem~\ref{T:contact-transverse}, $P$ is induced by a map which takes an $f\in Tr_\alpha(M,\mathcal F_N)$ onto the inverse foliation $f^{-1}\mathcal F_N$ and $G_*$ is induced by the bundle homomorphism $G:\nu\mathcal F_N\to \nu\Omega_{2q}$ covering $g$.\label{L:haefliger}
\end{lemma}
\begin{proof} We shall first show that the horizontal arrows in (\ref{Figure:Haefliger}) are well defined.
If $f\in Tr_{\alpha}(M,\mathcal{F}_N)$ then the inverse foliation $f^*\mathcal F_N$ belongs to $Fol^{2q}_\alpha(M)$. Furthermore, if $f_t$ is a homotopy in $Tr_{\alpha}(M,\mathcal{F}_N)$, then the map $F:M\times\I\to N$ defined by $F(x,t)=f_t(x)$ is clearly transversal to $\mathcal F_N$ and so $\tilde{\mathcal F}=F^*\mathcal F_N$ is a foliation on $M\times\I$.
The restriction of $\tilde{\mathcal F}$ to $M\times\{t\}$ is the same as the foliation $f^*_t(\mathcal F_N)$, which is a contact foliation subordinate to $\alpha$. Hence, we get a map \[\pi_0(Tr_{\alpha}(M,\mathcal{F}_N))\stackrel{P}\longrightarrow \pi_0(Fol^{2q}_{\alpha}(M))\] defined by \[[f]\longmapsto [f^{-1}\mathcal{F}_N]\]
where $[f^{-1}\mathcal{F}_N]$ denotes the integrable homotopy class of the foliation $f^{-1}\mathcal{F}_N$. On the other hand, since $g:N\to B\Gamma_{2q}$ classifies the foliation $\mathcal F_N$, there is a vector bundle homomorphism $G:\nu\mathcal F_N\to \nu\Omega_{2q}$ covering $g$. This induces a map
\[G_*: \pi_0(\mathcal E_\alpha(TM,\nu(\mathcal F_N)))\to  \pi_0(\mathcal E_\alpha(TM,\nu\Omega_{2q}))\] which takes an element $[F]\in \mathcal E_\alpha(TM,\nu(\mathcal F_N))$ onto $[G\circ F]$.
We now prove the commutativity of (\ref{Figure:Haefliger}). Note that if $f\in Tr_{\alpha}(M,\mathcal{F}_N))$ then $g\circ f:M\to B\Gamma_{2q}$ classifies the foliation $f^*\mathcal F_N$. Let $\widetilde{df}:\nu(f^*\mathcal F_N)\to \nu(\mathcal F_N)$ be the unique map making the following diagram commutative:
 \[
 \xymatrix@=2pc@R=2pc{
TM\ar@{->}[r]^-{df}\ar@{->}[d]_-{\pi_M} & TN\ar@{->}[d]^-{\pi_N}\\
\nu (f^*\mathcal{F}_N)\ar@{->}[r]_{\widetilde{df}} & \nu(\mathcal F_N)
}
\]
where $\pi_M:TM\to\nu(f^*\mathcal F_N)$ is the projection map onto the normal bundle of $f^*\mathcal F_N$.
Observe that $G\circ\widetilde{df}:\nu(f^*\mathcal F_N)\to \nu(\Omega_{2q})$ covers the map $g\circ f$ and $(G\circ \widetilde{df},g\circ f)$ is a pullback diagram. Therefore, we have
\[H_\alpha([f^*\mathcal F_N])=[(G\circ\widetilde{df})\circ \pi_M]=[G\circ(\pi\circ df)].\]
This proves the commutativity of (\ref{Figure:Haefliger}).\end{proof}

{\em Proof of Theorem ~\ref{haefliger_contact}}. The proof is exactly similar to that of Haefliger's classification theorem. The main idea is to reduce the classification to Theorem~\ref{T:contact-transverse} by using Theorem~\ref{HL} and Lemma~\ref{L:haefliger}. We refer to \cite{francis} for a detailed proof of Haefliger's theorem.\qed\\

\begin{theorem}Let $(M,\alpha)$ be an open contact manifold and let $\tau:M\to BU(n)$ be a map classifying the symplectic vector bundle $\xi=\ker\alpha$. Then there is a bijection between the elements of $\pi_0(\mathcal E_{\alpha}(TM,\nu\Omega))$ and the homotopy classes of triples $(f,f_0,f_1)$, where $f_0:M\to BU(q)$, $f_1:M\to BU(n-q)$ and $f:M\to B\Gamma_{2q}$ such that
\begin{enumerate}\item $(f_0,f_1)$ is homotopic to $\tau$ in $BU(n)$ and
\item $Bd\circ f$ is homotopic to $Bi\circ f_0$ in $BGL_{2q}$.\end{enumerate}
In other words the following diagrams are homotopy commutative:\\
\[\begin{array}{ccc}
\xymatrix@=2pc@R=2pc{
& &\ \ B\Gamma(2q)\ar@{->}[d]^{Bd}\\
M \ar@{->}[r]_-{f_0}\ar@{-->}[urr]^{f} & BU(q)\ar@{->}[r]_{Bi} & BGL(2q)
}
& \hspace{1cm}&
\xymatrix@=2pc@R=2pc{
&\ \ BU(q)\times BU(n-q)\ar@{->}[d]^{\oplus}\\
M \ar@{->}[r]_-{\tau}\ar@{-->}[ur]^{(f_0,f_1)}& BU(n)
}\end{array}\]
\end{theorem}

\begin{proof}
An element $(F,f)\in \mathcal E_{\alpha}(TM,\nu\Omega)$ defines a (symplectic) splitting of the bundle $\xi$ as
\[\xi \cong (\ker F\cap \xi)\oplus (\ker F\cap \xi)^{d'\alpha}\]
since $\ker F\cap \xi$ is a symplectic subbundle of $\xi$. Let $F'$ denote the restriction of $F$ to $(\ker F\cap \xi)^{d'\alpha}$. It is easy to see that $(F',f):(\ker F\cap \xi)^{d'\alpha}\to \nu(\Omega)$ is a vector bundle map which is fibrewise isomorphism. If $f_0:M\to BU(q)$ and $f_1:M\to BU(n-q)$ are continuous maps classifying the vector bundles $\ker F\cap \xi$ and $(\ker F\cap \xi)^{d'\alpha}$ respectively, then the classifying map $\tau$ of $\xi$ must be homotopic to $(f_0,f_1):M\to BU(q)\times BU(n-q)$ in $BU(n)$ (Recall that the isomorphism classes of Symplectic vector bundles are classified by homotopy classes of continuous maps into $BU$ \cite{husemoller}). Furthermore, note that $(\ker F\cap \xi)^{d'\alpha}\cong f^*(\nu\Omega)=f^*(Bd^*EGL_{2q}(\R))$; therefore, $Bd\circ f$ is homotopic to $f_0$ in $BGL(2q)$.

Conversely, take a triple $(f,f_0,f_1)$ such that
\[Bd\circ f\sim Bi\circ f_0 \text{ and } (f_0,f_1)\sim \tau.\]
Then $\xi$ has a symplectic splitting given by $f_0^*EU(q)\oplus f_1^*EU(n-q)$. Further, since $Bd\circ f\sim Bi\circ f_0$, we have $f_0^*EU(q)\cong f^*\nu(\Omega)$. Hence there is an epimorphism $F:\xi\stackrel{p_2}{\longrightarrow} f_0^*EU(q) \cong f^*\nu(\Omega)$ whose kernel $f_1^*EU(n-q)$ is a symplectic subbundle of $\xi$. Finally, $F$ can be extended to an element of $\mathcal E_\alpha(TM,\nu\Omega)$ by defining its value on $R_\alpha$ equal to zero.\end{proof}

\begin{definition}{\em Let $N$ be a contact submanifold of $(M,\alpha)$ such that $T_xN$ is transversal to $\xi_x$ for all $x\in N$. Then $TN\cap \xi|_N$ is a symplectic subbundle of $\xi$. The symplectic complement of $TN\cap \xi|_N$ with respect to $d'\alpha$ will be called \emph{the normal bundle of the contact submanifold $N$}.}
\end{definition}
The following result is a direct consequence of the above classification theorem.
\begin{corollary} Let $B$ be a symplectic subbundle of $\xi$ with a classifying map $g:M\to BU(q)$. The integrable homotopy classes of contact foliations on $M$ with their normal bundles isomorphic to $B$ are in one-one correspondence with the homotopy classes of lifts of $Bi\circ g$ in $B\Gamma_{2q}$.
\end{corollary}

We end this article with an example to show that a contact foliation on a contact manifold need not be transversally symplectic, even if its normal bundle is a symplectic vector bundle.
\begin{definition}{\em (\cite{haefliger1}) A codimension ${2q}$-foliation $\mathcal F$ on a manifold $M$ is said to be \emph{transverse symplectic} if $\mathcal F$ can be represented by Haefliger cocycles which take values in the groupoid of local symplectomorphisms of $(\R^{2q},\omega_0)$.}
\end{definition}
Thus the normal bundle of a transversally symplectic foliation has a symplectic structure. It can be shown that if $\mathcal F$ is transversally symplectic then there exists a closed 2-form $\omega$ on $M$ such that $\omega^q$ is nowhere vanishing and $\ker\omega=T\mathcal F$.

\begin{example}
{\em
Let us consider a closed almost-symplectic manifold $V^{2n}$ which is not symplectic (e.g., we may take $V$ to be $\mathbb S^6$) and let $\omega_V$ be a non-degenerate 2-form on $V$ defining the almost symplectic structure. Set $M=V\times\mathbb{R}^3$ and let $\mathcal{F}$ be the foliation on $M$ defined by the fibres of the projection map $\pi:M\to V$. Thus the leaves are $\{x\}\times\mathbb{R}^3,\ x\in V$. Consider the standard contact form $\alpha=dz+x dy$ on the Euclidean space $\R^3$ and let $\tilde{\alpha}$ denote the pull-back of $\alpha$ by the projection map $p_2:M\to\R^3$. The 2-form $\beta=\omega_V\oplus d\alpha$ on $M$ is of maximum rank and it is easy to see that $\beta$ restricted to $\ker\tilde{\alpha}$ is non-degenerate. Therefore $(\tilde{\alpha},\beta)$ is an almost contact structure on $M$. Moreover, $\tilde{\alpha}\wedge \beta|_{T\mathcal{F}}$ is nowhere vanishing.

We claim that there exists a contact form $\eta$ on $M$ such that its restrictions to the leaves of $\mathcal F$ are contact.
Recall that there exists a surjective map \[(T^*M)^{(1)}\stackrel{D}{\rightarrow}\wedge^1T^*M \oplus \wedge^2T^*M\] such that $D\circ j^1(\alpha)=(\alpha,d\alpha)$ for any 1-form $\alpha$ on $M$. Let
\[r:\wedge^1T^*M \oplus \wedge^2T^*M\rightarrow \wedge^1T^*\mathcal{F} \oplus \wedge^2T^*\mathcal{F}\] be the restriction map defined by the pull-back of forms and let $A\subset \Gamma(\wedge^1T^*M \oplus \wedge^2T^*M)$ be the set of all pairs $(\eta,\Omega)$ such that $\eta \wedge \Omega^{n+1}$ is nowhere vanishing and let $B\subset \Gamma(\wedge^1T^*\mathcal{F} \oplus\wedge^2T^*\mathcal{F})$ be the set of all pairs whose restriction on $T\mathcal{F}$ is nowhere vanishing. Now set $\mathcal{R}\subset (T^*M)^{(1)}$ as
\[\mathcal{R}=D^{-1}(A)\cap (r\circ D)^{-1}(B).\] Since both $A$ and $B$ are open so is $\mathcal{R}$. Now if we consider the fibration $M\stackrel{\pi}{\rightarrow}V$ then it is easy to see that the diffeotopies of $M$ preserving the fibers of $\pi$ sharply moves $V\times 0$ and $\mathcal{R}$ is invariant under the action of such
diffeotopies. So by Theorem~\ref{T:gromov-invariant} there exists a contact form $\eta$ on $Op(V\times 0)=V\times\mathbb{D}^3_{\varepsilon}$ for some $\varepsilon>0$, and $\eta$ restricted to each leaf of the foliation $\mathcal F$ is also contact. Now take a diffeomorphism $g:\mathbb{R}^3\rightarrow \mathbb{D}^3_{\varepsilon}$. Then  $\eta'=(id_V\times g)^*\eta$ is a contact form on $M$. Further, $\mathcal{F}$ is a contact foliation relative to $\eta'$ since $id_V\times g$ is foliation preserving.

But $\mathcal{F}$ can not be transversal symplectic because then there would exist a closed 2-form $\beta$ whose restriction to $\nu \mathcal{F}=\pi^*(TV)$ would be non-degenerate. This would imply that $V$ is a symplectic manifold contradicting our hypothesis.}

\end{example}

\section{Examples of contact foliations on contact manifolds}

The odd dimensional spheres $\mathbb S^{2n+1}$, $n\geq 1$, are examples of contact manifolds as described in Example~\ref{ex:contact}. We shall show that the open submanifolds of $\mathbb S^{n+1}$ obtained by deleting a lower dimensional sphere from it admits contact foliations.
We shall first interpret Corollary~\ref{T:contact_submersion} in terms of certain $2n$-frames in $M$, when the target manifold is an Euclidean space.
Recall from Section 2 that the tangent bundle $TM$ of a contact manifold $(M,\alpha)$ splits as $\ker\alpha\oplus\ker \,d\alpha$. Let $P:TM\to\ker\alpha$ be the projection morphism onto $\ker\alpha$ relative to this splitting. We shall denote the projection of a vector field $X$ on $M$ under $P$  by $\bar{X}$. For any smooth function $h:M\to \R$, $X_h$ will denote the contact Hamiltonian vector field defined as in the prelimiaries (see equations (\ref{contact_hamiltonian1})).
\begin{lemma} Let $(M,\alpha)$ be a contact manifold and $f:M\to \R^{2n}$ be a submersion with coordinate functions $f_1,f_2,\dots,f_{2n}$. Then the following statements are equivalent:
\begin{enumerate}\item[(C1)] $f$ is a contact submersion.
\item[(C2)] The restriction of $d\alpha$ to the bundle spanned by $X_{f_1},\dots,X_{f_{2n}}$ defines a symplectic structure.
\item[(C3)] The vector fields $\bar{X}_{f_1},\dots,\bar{X}_{f_{2n}}$ span a symplectic subbundle of $(\xi,d'\alpha)$.
\end{enumerate}\end{lemma}
\begin{proof} If $f:(M,\alpha)\to\R^{2n}$ is a contact submersion then the following relation holds pointwise:
\begin{equation}\ker df\cap \ker\alpha=\langle \bar{X}_{f_1},...,\bar{X}_{f_{2n}}\rangle^{\perp_{d'\alpha}},\end{equation}
where the right hand side represents the symplectic complement of the subbundle spanned by $\bar{X}_{f_1},...,\bar{X}_{f_{2n}}$ with respect to $d'\alpha$. Indeed, for any $v\in \ker\alpha$,
\[  d'\alpha(\bar{X}_{f_i},v)=-df_i(v),\ \ \text{ for all }i=1,...,2n \]
Therefore, $v\in\ker\alpha\cap\ker df$ if and only if $d'\alpha(\bar{X}_{f_i},v)=0$ for all $i=1,\dots,2n$, that is $v\in \langle \bar{X}_{f_1},...,\bar{X}_{f_{2n}}\rangle^{\perp_{d'\alpha}}$. Thus, the equivalence of (C1) and (C3) is a consequence of the equivalence between (S1) and (S2). The equivalence of (C2) and (C3) follows from the relation  $d\alpha(X,Y)=d\alpha(\bar{X},\bar{Y})$, where $X,Y$ are any two vector fields on $M$. \end{proof}

An ordered set of vectors $e_{1}(x),...,e_{2n}(x)$ in $\xi_x$ will be called a \emph{symplectic $2n$-frame} \index{symplectic $2n$-frame} in $\xi_x$ if the subspace spanned by these vectors is a symplectic subspace of $\xi_x$ with respect to the symplectic form $d'\alpha_x$. Let $T_{2n}\xi$ be the bundle of symplectic $2n$-frames in $\xi$ and $\Gamma(T_{2n}\xi)$ denote the space of sections of $T_{2n}\xi$ with the $C^{0}$ compact open topology.

For any smooth submersion $f:(M,\alpha)\rightarrow \mathbb{R}^{2n}$, define the \emph{contact gradient} of $f$ by
\[\Xi f(x)=(\bar{X}_{f_{1}}(x),...,\bar{X}_{f_{2n}}(x)),\]
where $f_{i}$, $i=1,2,\dots,2n$, are the coordinate functions of $f$. If $f$ is a contact submersion then $\bar{X}_{f_{1}}(x),...,\bar{X}_{f_{2n}}(x))$ span a symplectic subspace of $\xi_x$ for all $x\in M$, and hence $\Xi f$ becomes a section of $T_{2n}\xi$.

\begin{theorem}
\label{ED}
Let $(M^{2m+1},\alpha)$ be an open contact manifold. Then the contact gradient map $\Xi:\mathcal{C}_\alpha(M,\mathbb{R}^{2n})\rightarrow \Gamma(T_{2n}\xi)$ is a weak homotopy equivalence.
\end{theorem}
\begin{proof} As $T\mathbb{R}^{2n}$ is a trivial vector bundle, the map
\[i_{*}:\mathcal{E}_\alpha(TM,\mathbb{R}^{2n})\rightarrow \mathcal{E}_\alpha(TM,T\mathbb{R}^{2n})\]
induced by the inclusion $i:0 \hookrightarrow \mathbb{R}^{2n}$ is a homotopy equivalence, where $\mathbb{R}^{2n}$ is regarded as the vector bundle over $0\in \mathbb{R}^{2n}$. The homotopy inverse $c$ is given by the following diagram. For any $F\in \mathcal E_\alpha(TM,T\R^{2n})$, $c(F)$ is defined by as $p_2\circ F$,
\[\begin{array}{ccccc}
TM & \stackrel{F}{\longrightarrow} & T\mathbb{R}^{2n}=\mathbb{R}^{2n}\times \mathbb{R}^{2n} & \stackrel{p_2}{\longrightarrow} & \mathbb{R}^{2n}\\
\downarrow & & \downarrow & & \downarrow\\
M & \longrightarrow & \mathbb{R}^{2n} & \longrightarrow & 0
\end{array}\]
where $p_2$ is the projection map onto the second factor.

Since $d'\alpha$ is non-degenerate, the contraction of $d'\alpha$ with a vector $X\in\ker\alpha$ defines an isomorphism
\[\phi:\ker\alpha \rightarrow (\ker\alpha)^*.\]
We define a map $\sigma:\oplus_{i=1}^{2n}T^*M\to \oplus_{i=1}^{2n}\xi$ by
\[\sigma(G_1,\dots,G_{2n})=-(\phi^{-1}(\bar{G}_1),...,\phi^{-1}(\bar{G}_{2n})),\]
where $\bar{G}_i=G_i|_{\ker\alpha}$. Then noting that
\[\ker(G_1,\dots,G_{2n})\cap \ker\alpha=\langle\phi^{-1}(\bar{G}_1),\dots,\phi^{-1}(\bar{G}_{2n})\rangle^{\perp_{d'\alpha}},\]
we get a map $\tilde{\sigma}$ by restricting $\sigma$ to $\mathcal E(TM,\R^{2n})$:
\[\tilde{\sigma}:{\mathcal E}(TM,\mathbb{R}^{2n})\longrightarrow \Gamma(M,T_{2n}\xi),\]
Moreover, the contact gradient map $\Xi$ factors as $\Xi= \tilde{\sigma} \circ c \circ d$:
\begin{equation}\mathcal{C}_\alpha(M,\mathbb{R}^{2n})\stackrel{d}\rightarrow \mathcal{E}_\alpha(TM,T\mathbb{R}^{2n})\stackrel{c}\rightarrow \mathcal{E}_\alpha(TM,\mathbb{R}^{2n})\stackrel{\tilde{\sigma}}\rightarrow \Gamma(T_{2n}\xi).\end{equation}
To see this take any $f:M\to \R^{2n}$. Then, $c(df)=(df_{1},...,df_{2n})$, and hence
\[ \tilde{\sigma} c (df)=(\phi^{-1}(df_1|_\xi),...,\phi^{-1}(df_{2n}|_\xi)) = (\bar{X}_{f_1},\dots,\bar{X}_{f_{2n}})=\Xi(f)\]
which gives $\tilde{\sigma} \circ c \circ d(f)=\Xi f$.

We claim that $\tilde{\sigma}: \mathcal{E}_\alpha(TM,\mathbb{R}^{2n})\to \Gamma(T_{2n}\xi)$ is a homotopy equivalence.
To prove this we define a map $\tau: \oplus_{i=1}^{2n}\xi \to  \oplus_{i=1}^{2n} T^*M$
by the formula \[\tau(X_1,\dots,X_{2n})=(i_{X_1}d\alpha,...,i_{X_{2n}} d\alpha)\]
which induces a map $\tilde{\tau}: \Gamma(T_{2n}\xi) \to  \mathcal{E}(TM,\mathbb{R}^{2n})$.
It is easy to verify that $\tilde{\sigma} \circ \tilde{\tau}=id$. In order to show that $\tilde{\tau}\circ\tilde{\sigma}$ is homotopic to the identity, take any $G\in  \mathcal E_\alpha(TM,\R^{2n})$ and let $\widehat{G}=(\tau\circ \sigma)(G)$. Then $\widehat{G}$ equals $G$ on $\ker\alpha$. Define a homotopy between $G$ and $\hat{G}$ by $G_t=(1-t)G+t\widehat{G}$. Then $G_t=G$ on $\ker\alpha$ and hence $\ker G_t\cap \ker\alpha=\ker G\cap \ker\alpha$. This also implies that each $G_t$ is an epimorphism. Thus, the homotopy $G_t$ lies in $\mathcal E_\alpha(TM,\R^{2n})$. This shows that $\tilde{\tau}\circ \tilde{\sigma}$ is homotopic to the identity map.

This completes the proof of the theorem since $d:\mathcal{C}(M,\mathbb{R}^{2n}) \rightarrow \mathcal{E}(TM,T\mathbb{R}^{2n})$ is a weak homotopy equivalence (Theorem~\ref{T:contact-transverse}) and $c$, $\tilde{\sigma}$ are homotopy equivalences.\end{proof}

\begin{example}
{\em Let $\mathbb{S}^{2N-1}$ denote the $2N-1$ sphere in $\R^{2N}$
\[\mathbb{S}^{2N-1}=\{(z_{1},...,z_{2N})\in \mathbb{R}^{2N}: \Sigma_{1}^{2N}|z_{i}|^{2}=1\}\]
This is a standard example of a contact manifold where the contact form $\eta$ is induced from the 1-form $\sum_{i=1}^{N} (x_i\,dy_i-y_i\,dx_i)$ on $\R^{2N}$. For $N>K$, we consider the open manifold $\mathcal S_{N,K}$ obtained from $\mathbb{S}^{2N-1}$ by deleting a $(2K-1)$-sphere:
\begin{center}$\mathcal{S}_{N,K}=\mathbb S^{2N-1}\setminus \mathbb{S}^{2K-1}$,\end{center} where
\[\mathbb{S}^{2K-1}=\{(z_{1},...,z_{2K},0,...,0)\in \mathbb{R}^{2N}: \Sigma_{1}^{2K}|z_{i}|^{2}=1\}\]
Then $\mathcal{S}_{N,K}$ is an contact submanifold of $\mathbb S^{2N-1}$. Let $\xi$ denote the contact structure associated to the contact form $\eta$ on $\mathcal S_{N,K}$. Since $\xi\to \mathcal S_{N,K}$ is a symplectic vector bundle, we can choose a complex structure $J$ on $\xi$ such that $d'\eta$ is $J$-invariant. Thus, $(\xi,J)$ becomes a complex vector bundle of rank $N-1$.

We define a homotopy $F_t:\mathcal S_{N,K}\to \mathcal S_{N,K}$, $t\in [0,1]$, as follows: For $(x,y)\in \mathbb{R}^{2k}\times \mathbb{R}^{2(N-k)}\cap \mathcal{S}_{N,K}$
\[F_t(x,y)=\frac{(1-t)(x,y)+t(0,y/\|y \|)}{\|(1-t)(x,y)+t(0,y/\| y \|) \|}\]
This is well defined since $y\neq 0$. It is easy to see that $F_0=id$, $F_1$ maps $\mathbb{S}^{2(N-K)-1}$ into $\mathcal S_{N,K}$ and  the homotopy fixes $\mathbb{S}^{2(N-K)-1}$ pointwise. Define $r:\mathcal S_{N,K}\rightarrow \{0\}\times \mathbb{R}^{2(N-k)}\cap \mathcal S_{N,K}\mathbb{S}^{2(N-K)-1}\simeq \mathbb{S}^{2(N-K)-1}$ by
\[r(x,y)= (0,y/\|y\|), \ \ \ (x,y)\in\R^{2K}\times\R^{2(N-K)}\cap \mathcal{S}_{N,K}\]
Then $F_1$ factors as $F_1=i\circ r$, where $i$ is the inclusion map, and we have the following diagram:
\[
\begin{array}{lcccl}
  r^*(i^*\xi)&\longrightarrow&i^*\xi&\longrightarrow&\xi\\
\downarrow&&\downarrow&&\downarrow\\
\mathcal{S}_{N,K}&\stackrel{r}{\longrightarrow}&\mathbb{S}^{2(N-K)-1}&\stackrel{i}{\longrightarrow}&\mathcal{S}_{N,K}\end{array}\]
Hence, $\xi=F_0^*\xi\cong F_1^*\xi=r^*(\xi|_{S^{(2N-2K)-1}})$ as complex vector bundles.
Since $\xi$ is a (complex) vector bundle of rank $N-1$, $\xi|_{\mathbb S^{2(N-K)-1}}$ will have a decomposition of the following form (\cite{husemoller}):
\[\xi|_{S^{(2N-2K)-1}}\cong \tau^{N-K-1}\oplus \theta^K,\]
where $\theta^K$ is a trivial complex vector bundle of rank $K$ and $\tau^{N-K-1}$ is a complementary subbundle. Hence $\xi$ must also have a trivial direct summand $\theta$ of rank $K$. Moreover, $\theta$ will be a symplectic subbundle of $\xi$ since the complex structure $J$ is compatible with the symplectic structure $d'\eta$ on $\xi$. Thus, $S_{N,K}$ admits a symplectic $2K$ frame spanning $\theta$. Hence, by Theorem~\ref{ED}, there exist contact submersions of $\mathcal S_{N,K}$ into $\R^{2K}$. Consequently, $\mathcal S_{N,K}$ admits contact foliations of codimension $2K$ for each $K<N$.
}\end{example}


\end{document}